\documentclass[oneside,english,intlimits]{amsart}
\usepackage[T1]{fontenc}
\usepackage[latin9]{inputenc}
\usepackage{wrapfig}
\usepackage{amsthm}
\usepackage{amstext}
\usepackage{graphicx}
\usepackage{amssymb}

\makeatletter
\numberwithin{equation}{section}
\numberwithin{figure}{section}
\theoremstyle{plain}
\newtheorem{thm}{Theorem}[section]
  \theoremstyle{plain}
  \newtheorem{assumption}[thm]{Assumption}
  \theoremstyle{plain}
  \newtheorem{cor}[thm]{Corollary}
  \theoremstyle{remark}
  \newtheorem{rem}[thm]{Remark}
  \theoremstyle{plain}
  \newtheorem{prop}[thm]{Proposition}
  \theoremstyle{plain}
  \newtheorem{lem}[thm]{Lemma}
  \theoremstyle{plain}
  \newtheorem{fact}[thm]{Fact}
  \theoremstyle{definition}
  \newtheorem{defn}[thm]{Definition}

\numberwithin{equation}{subsection}
\numberwithin{figure}{section}

\usepackage{graphicx}
\usepackage{xcolor}
\definecolor{darkblue}{rgb}{0,0,0.5} 
\usepackage{transparent}
\usepackage[adobe-utopia]{mathdesign}

\makeatother

\usepackage{babel}

\begin{document}

\title[Curvatures of self-conformal sets]{Fractal curvatures and Minkowski content of self-conformal sets}

\author{Tilman Johannes Bohl}

\thanks{Supported by grant DFG ZA 242/5-1. The author has previously worked
under the name Tilman Johannes Rothe.}

\email{Tilman.Bohl@uni-jena.de}

\address{Mathematical Institute, Friedrich Schiller University Jena, Germany.}



\maketitle
\newcommand\comment[1]{}
\newcommand\printornot[1]{#1}
\newcommand\MpOrRaster[2]{#2}

\global\long\def\op#1{\operatorname{#1}}
\global\long\def\Int#1{\overset{\circ}{#1}}
\global\long\def\Frac{\op{frac}}
\global\long\def\esup{\op{ess\, sup}}
\global\long\def\var{\op{var}}
\global\long\def\diam{\op{diam}}
\global\long\def\orth{\operatorname{orth}}
\global\long\def\id{\operatorname{id}}
\global\long\def\dist{\operatorname{dist}}
\global\long\def\assign{\mathrel{:=}}
\global\long\def\G{G}
\global\long\def\H{\mathcal{H}}
\global\long\def\Hl{\H_{\G}}
\global\long\def\ra{\rightarrow}
\global\long\def\Rd{\mathbb{R}^{d}}
\global\long\def\mJ{F}
\global\long\def\mI{I}
\global\long\def\mIN{\mI^{\mathbb{N}}}
\global\long\def\mX{X}
\global\long\def\bV{V}
\global\long\def\mV{\bV_{MU}}
\global\long\def\mK{K}
\global\long\def\mfsym{\phi}
\global\long\def\Kf{\mK_{\mfsym}}
\global\long\def\mf#1{\mfsym_{#1}}
\global\long\def\apsym{\mbox{\rotatebox[origin=c]{180}{\ensuremath{\psi}}}}
\global\long\def\bpsym{\psi}
\global\long\def\ap#1#2{\apsym_{#1,#2}}
\global\long\def\bp#1#2{\bpsym_{#1,#2}}
\global\long\def\bps#1#2{\bpsym}
\comment{bps b/c unsure whether to abbreviate}\global\long\def\aplim#1#2{\ap{\widetilde{#1}}{#2}}
\global\long\def\bplim#1#2{\bp{\widetilde{#1}}{#2}}
\comment{p(word,point)(arg)}\global\long\def\ms{\sigma}
\global\long\def\msb{\ms_{\op{bi}}}
\global\long\def\zrall#1#2{\left\vert #1^{\prime}#2\right\vert }
\global\long\def\zr#1#2{\zrall{\mf{#1}}{#2}}
\global\long\def\zrap#1#2#3{\zrall{\ap{#1}{#2}}{#3}}
\global\long\def\zraplim#1#2#3{\zrap{\widetilde{#1}}{#2}{#3}}
\global\long\def\zrp#1#2#3{\zrall{\bp{#1}{#2}}{#3}}
\global\long\def\zrplim#1#2#3{\zrp{\widetilde{#1}}{#2}{#3}}
\global\long\def\mmc{\mu}
\global\long\def\mmi{\nu}
\global\long\def\mmbc{\mmc_{\op{bi}}}
\global\long\def\mmbi{\mmi_{\op{bi}}}
\global\long\def\icdens{p}
\global\long\def\icbdens{\icdens_{\op{bi}}}
\global\long\def\ze{r}
\global\long\def\Along{A}
\global\long\def\A#1#2#3{\Along\left(\left|#1-#2\right|,#3\right)}
\global\long\def\zD{D}
\global\long\def\R{R}
\global\long\def\Rvar{\epsilon}
\global\long\def\Rstart{\Rvar_{\max}}
\global\long\def\rmax{s_{\max}}
\global\long\def\rmin{\min_{i}\left\Vert \mf i^{\prime}\right\Vert }
\comment{R for renewal radii and Rvar for integration limits}\global\long\def\hi{q}
\comment{\hr,\hw integrand}\global\long\def\ex#1{\underline{#1}}
\comment{density \mmi over \mmc}\global\long\def\hw{\pi}
\global\long\def\hr{\rho}
\global\long\def\chunkh#1#2{\left(#1,#2\right]}
\global\long\def\chunk#1#2{1_{\left(#1,#2\right]}}
\global\long\def\entropy{H_{\mmi}}
\global\long\def\HG{\mathcal{H}_{\G}}
\global\long\def\holder{\gamma}
\global\long\def\myset{K}
\global\long\def\xhat{x}
\global\long\def\xplain{\hat{x}}
\global\long\def\rhat{\ze}
\global\long\def\rplain{\hat{\ze}}
\global\long\def\zhat{z}
\global\long\def\zplain{\hat{z}}
\global\long\def\yhat{y}
\global\long\def\yplain{\hat{y}}
\global\long\def\fibreavg#1#2{#1_{#2}^{\op{fib}}}

\begin{abstract}
For self-similar fractals, the Minkowski content and fractal curvature
have been introduced as a suitable limit of the geometric characteristics
of its parallel sets, i.e., of uniformly thin coatings of the fractal.
For some self-conformal sets, the surface and Minkowski contents are
known to exist. Conformal iterated function systems are more flexible
models than similarities. This work unifies and extends such results
to general self-conformal sets in $\Rd$. We prove the rescaled volume,
surface area, and curvature of parallel sets converge in a Cesaro
average sense of the limit. Fractal Lipschitz-Killing curvature-direction
measures localize these limits to pairs consisting of a base point
in the fractal and any of its normal directions. There is an integral
formula.

We assume only the popular Open Set Condition for the first-order
geometry, and remove numerous geometric assumptions. For curvatures,
we also assume regularity of the Euclidean distance function to the
fractal if the ambient dimension exceeds three, and a uniform integrability
condition to bound the curvature on {}``overlap sets''. A limited
converse shows the integrability condition is sharp. We discuss simpler,
sufficient conditions. 

The main tools are from ergodic theory. Of independent interest is
a multiplicative ergodic theorem: The distortion, how much an iterate
of the conformal iterated function system deviates from its linearization,
converges.
\end{abstract}

\section{Introduction}

\begin{wrapfigure}[17]{o}{0.37\textwidth}%
\printornot
{\MpOrRaster{\includegraphics[width=0.32\textwidth]{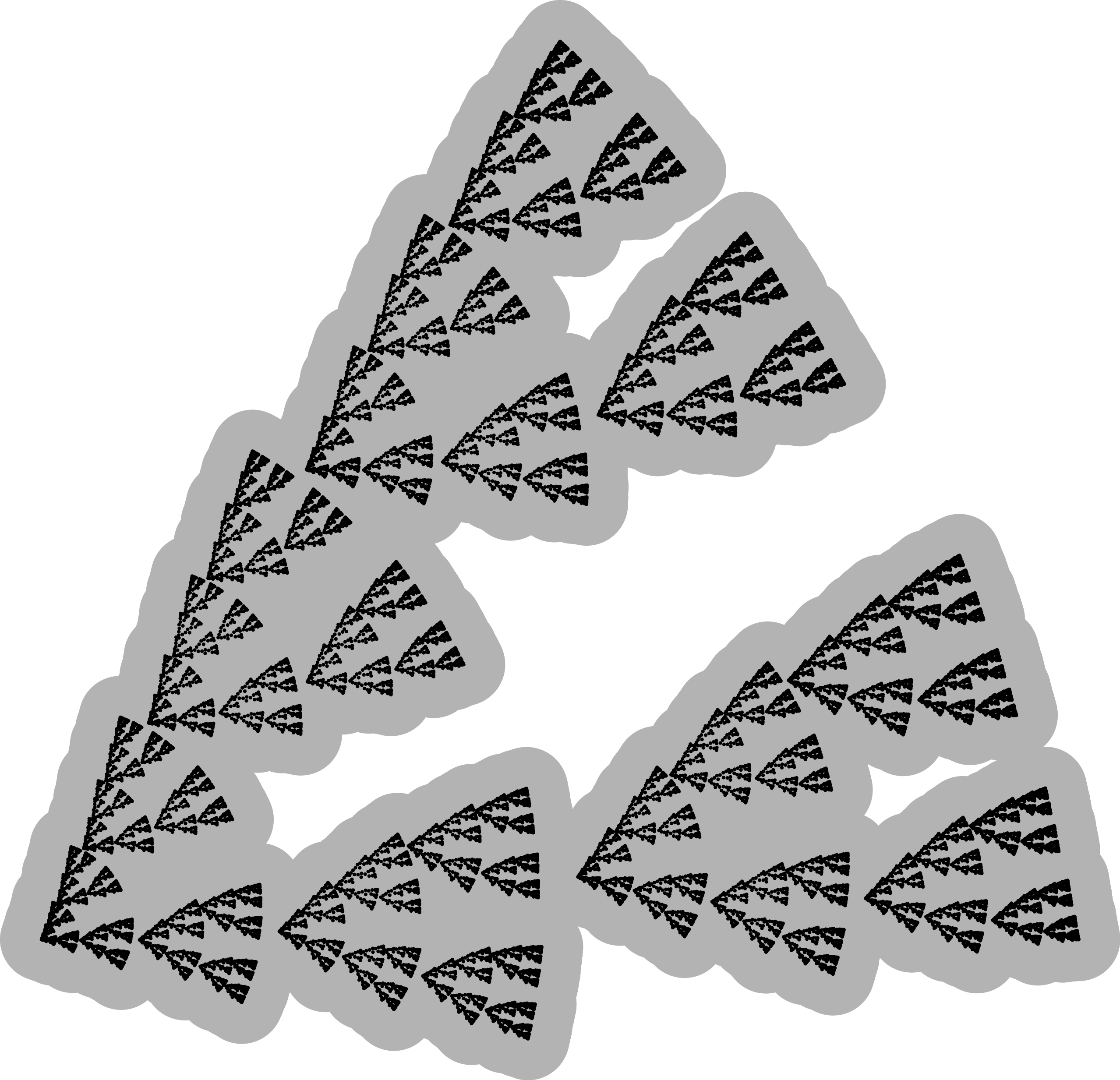}}{\includegraphics[width=0.32\textwidth]{conformal_parallel}}}\caption{~~~~~~~~~~~~~~~~~~~~~~~~~~~~~~~~~~~~~~~~~~~~~~~~~~~~~}
Dark: self-conformal set $\mJ$, IFS of three Möbius maps \eqref{eq:mobius transformation}.
Shaded: a parallel set $\mJ_{\ze}$.\end{wrapfigure}%
Let $\mJ\subseteq\Rd$ be a self-conformal set, i.e., invariant under
finitely many contractions $\mf i$ whose differentials $\mf i^{\prime}\left(x\right)$
are multiples of orthogonal matrices, \[
\mJ=\bigcup_{i\in\mI}\left\{ \mf ix\,:\, x\in\mJ\right\} .\phantom{\mJ=\bigcup_{i\in\mI}\left\{ \mf ix\,:\, x\in\mJ\right\} .}\]
 Any self-similar set $\mJ$ is a simple, {}``rigid'' special case.
We assume the well-known Open Set Condition to keep overlaps $\mf i\mJ\cap\mf j\mJ$
small.

Classical volume and curvature are not meaningfully applicable to
the geometry of $\mJ$. So we approximate $\mJ$ with thin {}``coatings'',
its parallel sets \[
\mJ_{\ze}\assign\left\{ x\,:\,\min_{y\in\mJ}\left|x-y\right|\leq\ze\right\} ,\,\ze>0.\phantom{\mJ_{\ze}\assign\left\{ x\,:\,\min_{y\in\mJ}\left|x-y\right|\leq\ze\right\} ,\,\ze>0.}\]
The limits of the rescaled Lebesgue measure and Lipschitz-Killing
curvature-direction measures of $\mJ_{\ze}$ do describe $\mJ$ geometrically.
The purpose of this work is to make that statement precise under minimal
assumptions.

The type of limit is most easily explained using the example of the
well-known Min\-kow\-ski content. We localize it as a measure version.
Denote by $C_{d}\left(\mJ_{\ze},A\right)$ the Lebesgue measure of
$\mJ_{\ze}\cap A$ (if $A\subseteq\Rd$). The rescaling factor $\ze^{\zD-d}$
keeps $\ze^{\zD-d}C_{d}\left(\mJ_{\ze},\cdot\right)$ from vanishing
trivially as $\ze\ra0$. (Thermodynamic formalism determines for the
correct dimension $\zD>0$, see below.) But the rescaled total volume
can oscillate. Cesaro averaging on a logarithmic scale suppresses
this. Without any further assumptions, we prove the localized average
Minkowski content converges (as a weak limit of measures):\begin{eqnarray*}
C_{d}^{\text{frac}}\left(\mJ,\cdot\right) & \assign & \lim_{\Rvar\ra0}\frac{1}{\left|\ln\Rvar\right|}\int_{\Rvar}^{1}\ze^{\zD-d}C_{d}\left(\mJ_{\ze},\cdot\right)\frac{d\ze}{\ze}\\
 & = & \lim_{T\ra\infty}\frac{1}{T}\int_{0}^{T}e^{-t\left(\zD-d\right)}C_{d}\left(\mJ_{e^{-t}},\cdot\right)dt.\end{eqnarray*}

Curvature and area fit into this approach as well as volume. Classically,
they are intrinsically determined and form a complete system of certain
Euclidean invariants. To illustrate classical curvature, let $\myset\subseteq\Rd$
be a compact body whose boundary $\partial\myset$ is a smooth submanifold.
Write $n\left(x\right)$ for the unit outer normal to $\myset$ at
$x\in\partial\myset$. The differential $n^{\prime}\left(x\right)$
of the Gauss map $n$ has real eigenvalues $\kappa_{1}\left(x\right)$,
$\dots$, $\kappa_{d-1}\left(x\right)$ (geometrically, inverse radii
of osculating spheres).%
\footnote{Curvature is second-order geometry because $n^{\prime}\left(x\right)$
is a second-order coordinate derivative, as opposed to Minkowski content
and various Hausdorff and packing measures.%
} They are called principal curvatures of $\myset$ at $x$. The full-dimensional
Hausdorff measure $\H^{d-1}$ coincides with the Riemannian volume
form on $\partial\myset$. Lipschitz-Killing curvature-direction measures
are integrals of symmetric polynomials of principal curvatures: For
integers $0\leq k<d$ and any Borel set $A\subseteq\Rd\times S^{d-1}$,
define \[
C_{k}\left(\myset,A\right)\assign\text{const}\int_{\partial\myset}1\left\{ \left(x,n\left(x\right)\right)\in A\right\} \,\sum_{j_{1}<\dots<j_{d-k-1}}\kappa_{j_{1}}\left(x\right)\cdots\kappa_{j_{d-k-1}}\left(x\right)\, d\H^{d-1}\left(x\right).\]
Our parallel sets $\mJ_{\ze}$ will only satisfy a positive reach
condition rather than smoothness. (Parallel sets do not depend on
the choice of $\mf i$ and enjoy widespread use in singular curvature
theory.) Their curvature-direction measures are defined via geometric
measure theory. Each point $x$ can have a different cone of normals.
Anisotropic curvature quantities also prove useful to describe heterogeneous
materials, see \cite{SchroederturkMecke11MinkowskiShapeAnalysis}
and the references therein. Therefore, we prefer curvature-direction
measures rather than non-directional Federer curvature, i.e., than
their $\Rd$ marginal measures. We prove the average fractal curvature
measures \begin{eqnarray*}
C_{k}^{\text{frac}}\left(\mJ,\cdot\right) & \assign & \lim_{\Rvar\ra0}\frac{1}{\left|\ln\Rvar\right|}\int_{\Rvar}^{1}\ze^{\zD-k}C_{k}\left(\mJ_{\ze},\cdot\right)\frac{d\ze}{\ze}\end{eqnarray*}
 converge under mild assumptions.

While most literature uses this normalization, a different one displays
more clearly $C_{k}^{\text{frac}}\left(\mJ,\cdot\right)$ does not
depend on the ambient dimension (Remark \ref{rem:Arguably renormalize Ck}
below).

We fix some notation for the assumptions. This dynamical system powers
all convergence results: The left shift on $\mJ$ inverts the generating
contractions, i.e., $\ms\left(\mf ix\right)\assign x$. Thermodynamic
formalism defines both a unique dimension $\zD>0$ and conformal probability
measure $\mmc$ on $\mJ$. The contraction ratio is its potential,
$\mmc\left(\mf{\omega_{m}}\circ\dots\circ\mf{\omega_{1}}\mJ\right)=\int_{\mJ}\zrall{\left(\mf{\omega_{m}}\circ\dots\circ\mf{\omega_{1}}\right)}{}^{\zD}d\mmc$,
$\omega_{k}\in\mI$. But the curvature is covariant under the similarity
$\mf i^{\prime}x$ rather than $\ms$. Also, only a similarity lets
us compare its image of $\mJ_{\ze}$ locally with any thinner parallel
set. (This restricts us to conformal $\mf i$.) The distortion $\bpsym$
captures the mismatch between iterations of the dynamical system and
their differentials (Figure \ref{fig:distortion and renewal radii}).
For an $\mI$-valued sequence $\omega$ and points $x,y\in\mJ$, we
prove convergence\begin{eqnarray*}
\bp{\widetilde{\omega|m}}xy & \assign & \phantom{{\lim_{m\ra\infty}}}\left(\left(\mf{\omega_{m}}\circ\dots\circ\mf{\omega_{1}}\right)^{\prime}x\right)^{-1}\left(\mf{\omega_{m}}\circ\dots\circ\mf{\omega_{1}}y-\mf{\omega_{m}}\circ\dots\circ\mf{\omega_{1}}x\right),\\
\bplim{\omega}xy & \assign & \lim_{m\ra\infty}\bp{\widetilde{\omega|m}}xy.\end{eqnarray*}

Two mild conditions guarantee $C_{k}^{\text{frac}}\left(\mJ,\cdot\right)$
converges (see Theorem \ref{thm:Main result}). Neither condition
applies to the volume ($k=d$), or to the surface area measure ($k=d-1$)
if restricted to $\Rd$ (instead of $\Rd\times S^{d-1}$).
\begin{enumerate}
\item Positive reach of the closed complement of almost every parallel set
both of $\mJ$ and of almost every distorted fractal $\bp{\omega}x\mJ$
(Assumption \ref{ass:Regularity-of-parallel-sets}). This guarantees
existence and continuity of curvature measures. It is automatically
satisfied in ambient dimensions $d\leq3$. In higher dimensions, it
is more convenient to check that almost all values of the Euclidean
distance function are regular.
\item A uniform integrability condition limits the $C_{k}^{\var}$ mass
on overlaps of images under different $\mf i$ (Assumption \ref{ass:(Uniform-integrability)}
below). This is sharp due to a limited converse (Remark \ref{rem:sharp}).
Let $B\left(x,\ze\right)$ be an $\ze$-ball around $x$. A sufficient,
integrability condition is that \[
\left(\ze,\omega,x\right)\mapsto\ze^{-k}\sup_{M>0}\frac{1}{M}\sum_{m=0}^{M-1}C_{k}^{\var}\left(\left(\bp{\widetilde{\omega\vert m}}x\mJ\right)_{\ze},\bp{\widetilde{\omega\vert m}}xB\left(x,a\ze\right)\right),\, a>1\]
 belongs to a Zygmund space related to $\mmc$. The much stronger,
sufficient condition \[
\op{ess\, sup}\limits _{x\in\mJ,\ze>0}\ze^{-k}C_{k}^{\var}\left(\mJ_{\ze},B\left(x,a\ze\right)\right)<\infty\]
intuitively means the rescaled principal curvatures should be uniformly
bounded.
\end{enumerate}
We are the first to prove average fractal curvature exists for general
self-conformal sets because our $C_{k}$ contains principal curvatures
when $k<d-1$. Our theorem encompasses all previous convergence results
about ($\zD$-dimensional) average Minkowski or surface content of
self-conformal sets \cite{KombrinkKessebohmer10OneDimSelfconformal,FreibergKombrink11MinkowskiConformalImages1109.3896,Kombrink11DissConformalMinkowski}.
We remove all previously needed geometrical conditions on the Open-Set-Condition
set or on the hole structure. Fractal curvatures or their total mass
were first investigated for self-similar sets \cite{Win08MR2423952,Zaehle11SelfsimRandomFractalsMR2763731,RatajZaehle10CurvatureDensitiesSelfSimilarSets,WinterZaehle10cmofsss,BohlZaehle12CurvatureDirectionMeasures1111.4457}.
The Minkowski content was treated earlier \cite{LapidusPomerance93MR1189091,Falconer95MinkowskiMeasurabilityMR1224615,Gatzouras00LacunarityMR1694290},
and in a very general context recently in \cite{RatajWinter09MeasuresOfParallelSetsArxiv09053279}
and \cite{RatajWinter11MinkowskiMeasurabilitySurface1111.1825}.

There is an integral formula for the limit $C_{k}^{\text{frac}}\left(\mJ,\cdot\right)$
(Theorem \ref{thm:Main result}). It is always finite. The formula
shows the Minkowski content is strictly positive. But for $k<d$,
the interpretation as fractal curvature should be checked if $C_{k}^{\text{frac}}\left(\mJ,\cdot\right)$
vanishes trivially \cite{Win08MR2423952}\comment{, \cite{PokornyWinter13Scaling}}.

Unsurprisingly, the fractal curvature is just as self-conformal as
$\mJ$. Its $\Rd$ marginal measure must therefore be a multiple of
the conformal probability $\mmc$. That means the fractal locally
looks the same almost everywhere. Its $S^{d-1}$ sphere marginal measure
inherits a group action invariance from the self-conformality ($\fibreavg f{\omega}$
\eqref{eq:DEF f limit} is invariant under $\G_{\eta u}$ \eqref{eq: DEF Brin group}
below). Otherwise, it is free to provide geometric information beyond
its total mass. For example, the Sierpinski gasket resembles an equilateral
triangle, and both have the same sphere marginals up to a constant
\cite{BohlZaehle12CurvatureDirectionMeasures1111.4457}. (But the
group spreads it uniformly over $S^{d-1}$ for generic $\mf i$, $\mI$.
Future research could ask how curvature approaches its limit \cite{LapidusPearseWinter11TubeFormulasArbitraryGeneratorsMR2799798}.)
In the self-similar case, $C_{k}^{\text{frac}}\left(\mJ,\cdot\right)$
is an (independent) product of both its marginal measures. In the
self-conformal case, it is a skew product. But the skewing factor
is easy to compute and only takes values in the orthogonal group of
$\Rd$ (the factor $\left(\bplim{\eta}u^{\prime}\hat{y}\right)^{\orth}$
in \eqref{eq:DEF f limit} below).

The motivation for this work is threefold. Fractal curvature is a
geometric parameter to distinguish between sets of identical dimensions.
Statistical estimates of various fractal dimensions already play a
role in the applications. It is natural to expect further geometric
parameters to be useful, both in their own right and to improve estimators
for dimensions. 

Secondly, curvature (geometry) determines analytic properties. In
the classical, smooth Riemannian case, the Dirichlet Laplacian eigenvalue
counting function has an expansion in terms of the boundary's curvatures.
Heat kernel estimates also involve curvature. In one dimension, the
next-to-leading order spectral asymptotics are proven: Let $A$ be
an open set whose boundary $\partial A$ has a nontrivial, $\zD$-dimensional
Minkowski content without Cesaro averaging (modified Weyl-Berry conjecture),
as $\lambda\ra\infty$: \[
\#\left\{ \text{eigenvalues }\leq\lambda\right\} =\text{const}C_{1}\left(A\right)\lambda^{1/2}+\text{const}_{\zD}C_{1}^{\op{frac}}\left(\partial A,A\times S^{1}\right)\lambda^{\zD/2}+o\left(\lambda^{\zD/2}\right).\]
 Upper and lower bounds hold at the $\lambda^{\zD/2}$ order if the
upper and lower Minkowski content are nontrivial \cite{LapidusPomerance93MR1189091}.
In several dimensions, true curvature exists, and the conjecture fails.
We hope to provide some missing geometric input. \comment{(The Laplacian
on the fractal itself is approximated in \cite{MoscoVivaldi07FractalSingularHomogenizationMR2323380,VivaldiCapitanelli11RandomKochLaplacianMR2812592,MoscoVivaldi09SierpinskiEnergyHomogenizationMR2533923,Vivaldi12HomogenizationSurvey,MoscoVivaldi12ThinFractalFibresGeneralCase}.)}

Thirdly, there is the long-standing quest in geometric measure theory
to extend the notion of curvature as far as possible. Classes of such
{}``classical'' sets include $C^{2}$-smooth manifolds, convex sets,
sets of positive reach, and their locally finite unions \cite{Federer59MR0110078,Zaehle86IntRepMR849863,RatajZaehle01MR1846894}.
See \cite{Bernig03AspectsCurvature} and the references therein for
an overview. Support measures generalize to all closed sets but loose
additivity and other characteristic features of curvature \cite{HugLastWeil04SupportMeasuresMR2031455}.
They and our fractal curvature usually live on mutually disjoint subsets
of $\mJ$, e.g. for the Sierpinski gasket.

The central ideas of the proof are explained in the respective section
headings. The methods are dynamical although the results are geometric.
We reduce the conformal case to similarities by separating $\mf{\,}$
into the distortion $\bpsym$ and its derivative $\mf{\,}^{\prime}$,
which matches Euclidean covariance. The normal directions necessitate
the two-sided shift dynamical system extended by the orthogonal group
co-cycle $\left(\mf{\,}^{\prime}\right)^{\orth}$.

\section{Preliminaries}

\subsection{Self-conformal sets and dynamics}

Conformal maps are automatically\textit{ (conjugate) holomorphic}
in $\mathbb{R}^{2}$ or even \textit{Möbius} in $\Rd$, $d>2$ (Liouville
theorem, e.g. \cite[Theorem 1.5]{Gehring92TopicsQuasiconformalMappingsMR1187087}).
\begin{assumption}
\label{ass:model}Throughout this entire paper, $\left\{ \mf i\,:\, i\in\mI\right\} $,
$2\leq\left|\mI\right|<\infty$ will be a conformal iterated function
system (IFS), i.e., each function \[
\mf i:\bV\ra\bV\]
 is an injective $C^{1+\holder}$-diffeomorphism defined on an open,
connected set $\bV\subseteq\Rd$ such that the differentials $\mfsym^{\prime}\left(x\right)$
are linear similarities of $\Rd$, and \[
\left|\mf ix-\mf iy\right|\leq\rmax\left|x-y\right|,\, x,y\in\bV\]
 for some global $\rmax<1$.  We assume the Open Set Condition (OSC):
There is an open, connected, bounded, nonempty set $\Int X$ such
that $\mX\subseteq\bV$, \[
\mf i\Int X\subseteq\Int X,\, i\in\mI\]
 and \[
\left(\mf i\Int X\right)\cap\left(\mf j\Int X\right)=\varnothing,\, i\neq j,\, i,j\in\mI.\]

\end{assumption}
Without loss of generality (\cite{PeresEA01ConformalStrongOpenSetConditionMR1838793}),
the set $\mX$ shall satisfy the \textit{Strong Open Set Condition}
\[
\Int X\cap\mJ\neq\varnothing.\]
The IFS shall extend conformaly in $C^{1+\holder}$ to the closure
of an open set $\mV\supseteq\overline{\bV}$: $\mf i:\overline{\mV}\ra\overline{\mV}$.

This is the situation studied in \cite{MauldinUrbanski96DimensionsMeasuresMR1387085},
except we allow only finitely many maps $\mf i$. Their cone condition
is not needed for finite $\mI$, and their bounded distortion condition
can be proved \cite{MauldinUrbanski03GraphDirectedMR2003772}. See
\cite{MauldinUrbanski96DimensionsMeasuresMR1387085} for further discussion
and facts given without a reference.

The\textit{ self-conformal fractal} is the unique, invariant, compact
set $\varnothing\neq\mJ=\bigcup_{i\in\mI}\mf i\mJ\subseteq\mX$. The
(Borel probability) \textit{conformal measure} $\mmc$ on $\mJ$ and
\textit{dimension} $\zD>0$ are uniquely characterized by \begin{equation}
\int_{\mJ}f\left(x\right)\, d\mmc\left(x\right)=\int_{\mJ}\sum_{i\in\mI}\zr ix^{\zD}f\left(\mf ix\right)\, d\mmc\left(x\right).\label{eq:perron frobenius, dimension, measure}\end{equation}
Abbreviate $x|n\assign x_{1}\dots x_{n}\in\mI^{n}$, the \textit{reversed
word} $\widetilde{x|n}\assign x_{n}\dots x_{1}$, $\mf{x|n}\assign\mf{x_{1}}\circ\dots\circ\mf{x_{n}}$,
$\mI^{*}\assign\bigcup_{n\in\mathbb{N}_{0}}\mI^{n}$. We will identify
$\mmc$-almost all points $x\in\mJ$ with their (unique) coding sequence
$x_{1}x_{2}\dots\in\mI^{-\mathbb{N}}$, i.e., $x\equiv\lim_{n\ra\infty}\mf{x|n}\left(v\right)$
for any starting $v\in\bV$. The \textit{(left) shift} $\ms:\mJ\rightarrow\mJ$
or $\ms:\mI_{\mJ}^{\mathbb{N}}\ra\mI_{\mJ}^{\mathbb{N}}$ maps $x$
to $\mf{x_{1}}^{-1}x$. The unique, equivalent, $\ms$-\textit{invariant
probability} $\mmi$ is ergodic. Denote its Hölder continuous density
$\icdens\assign d\mmi/d\mmc$. The proofs need the \textit{two-sided
shift space} $\mIN\times\mJ$. Let $\msb\left(\omega_{1}\omega_{2}\omega_{3}\dots,\, x_{1}x_{2}x_{3}\dots\right)\assign\left(x_{1}\omega_{1}\omega_{2}\dots,\, x_{2}x_{3}x_{4}\dots\right)$
for $\left(\omega,x\right)\in\mIN\times\mJ$; perhaps it is most natural
to think of $\omega$ in reverse order: $\msb\left(\left(\dots\omega_{2}\omega_{1}\right)\left(x_{1}x_{2}\dots\right)\right)=\left(\left(\dots\omega_{1}x_{1}\right)\left(x_{2}x_{3}\dots\right)\right)$.
Denote $\mmbc$, $\mmbi$, $\icbdens=d\mmbi/d\mmbc$ the Rokhlin extensions.

We extend it again, by the orthogonal group, to $\mIN\times\mJ\times O\left(d\right)$.
The new measure $\ex{\mmbi}\assign\mmbi\otimes\mathcal{H}_{O\left(d\right)}$
is the direct product with the Haar probability on $O\left(d\right)$.
The new left shift is \begin{equation}
\ex{\msb}^{m}\left(\omega,x,g\right)\assign\left(\widetilde{x|m}.\omega,\,\ms^{m}x,\,\left(\left(\mf{x|m}^{\prime}\ms^{m}x\right)^{\orth}\right)^{-1}g\right),\label{eq:DEF group extended shift}\end{equation}
 i.e., the \textit{skew product} with the \textit{Rokhlin co-cycle}
$\left(\omega,x\right)\mapsto\left(\mf{x|1}^{\prime}\ms x\right)^{\orth-1}$.

Recall some basic properties. The Perron-Frobenius operator gives
(\cite[Lemma 4.2.5]{PrzytyckiUrbanski10ConformalMR2656475}): \begin{equation}
\int_{\mJ}\frac{g}{\icdens}\left(x\right)d\mmi\left(x\right)=\int_{\mIN\times\mJ}\frac{g}{\icdens}\left(\mf{\widetilde{\omega\vert n}}x\right)d\mmbi\left(\omega,x\right),\, g\in L_{1}\left(\mmc\right).\label{eq:lem:(Perron-Frobenius-operator)}\end{equation}
Denote $B\left(x,\ze\right)\assign\left\{ y\,:\,\left\vert y-x\right\vert \leq\ze\right\} $.
$\mJ$ is a $\zD$-set for $\mmc$ (\textit{Ahlfors regular}): there
is a $0<C_{\mJ}<\infty$,

\begin{equation}
C_{\mJ}^{-1}\leq\frac{\mmc B\left(x,\ze\right)}{\ze^{\zD}}\leq C_{\mJ}.\label{eq:D-set-property Ahlfors}\end{equation}
\textit{Bounded distortion} is fundamental. There is a global constant
$0<\mK<\infty$ such that for all $x,y\in\overline{\mV}$, $\tau\in\mI^{*}$,
\begin{equation}
\mK^{-1}\leq\frac{\zr{\tau}x}{\zr{\tau}y}\leq\mK.\label{eq:K bounded distortion}\end{equation}
There is a constant $0<\Kf<\infty$ such that for all $x\in\mV$,
$\ze/\zr{\tau}x\leq\dist\left(x,\mV^{c}\right)$, $\tau\in\mI^{*}$,
\cite[(BDP.3)]{MauldinUrbanski96DimensionsMeasuresMR1387085}:\begin{equation}
B\left(\mf{\tau}x,\ze\Kf^{-1}\right)\subseteq\mf{\tau}B\left(x,\frac{\ze}{\zr{\tau}x}\right)\subseteq B\left(\mf{\tau}x,\ze\Kf\right).\label{eq:BDP.3}\end{equation}

The \textit{distortion}, how much the non-linear map $\mf{\omega|n}$
deviates from its differential (up to isometries), converges exponentially
(Proposition \ref{pro:(Distortion-converges)} below): \begin{eqnarray}
\bp{\widetilde{\omega|n}}xy & \assign & \phantom{\lim_{n\ra\infty}}\left(\mf{\widetilde{\omega|n}}^{\prime}x\right)^{-1}\left(\mf{\widetilde{\omega|n}}y-\mf{\widetilde{\omega|n}}x\right),\label{eq:def psi}\\
\bplim{\omega}xy & \assign & \lim_{n\ra\infty}\bp{\widetilde{\omega|n}}xy.\label{eq: def psi limit}\end{eqnarray}
 Note its derivative is point-wise a similarity. The \textit{Brin
ergodicity group} is the compact subgroup of $O\left(d\right)$ generated
by \begin{equation}
\G_{\eta u}\assign\overline{\left<\left(\left(\bplim{\eta}u^{\prime}x\right)^{\orth}\right)^{-1}\left(\mf i^{\prime}x\right)^{\orth}\left(\bplim{\eta}u^{\prime}\mf ix\right)^{\orth}\,:\, x\in\mJ,i\in\mI\right>_{O\left(d\right)}}.\label{eq: DEF Brin group}\end{equation}
Let $\HG$ be its Haar probability. For $f:\bV\times S^{d-1}\ra\mathbb{R}$,
the average of $f/\icdens$ over the \textit{ergodic fibre} through
$\left(\omega,x\right)$ is \begin{equation}
\fibreavg f{\omega}\left(z,n\right)\assign\int\limits _{\mJ}\int\limits _{\G_{\eta u}}f\left(\hat{y},\left(\bplim{\eta}u^{\prime}\hat{y}\right)^{\orth-1}\hat{g}\left(\bplim{\omega}u^{\prime}z\right)^{\orth}n\right)\, d\HG\left(\hat{g}\right)d\mmc\left(\hat{y}\right).\label{eq:DEF f limit}\end{equation}
Different choices of $\left(\eta,u\right)\in\mIN\times\mJ$ do not
alter $\fibreavg f{\omega}$ (but do conjugate $\G_{\eta u}$). So
$\G_{\eta u}$ and $\fibreavg f{\omega}$ can be skipped when replacing
$u$ with $x$. Readers not interested in directional $C_{k}$ should
ignore all underbars and set $\G_{\eta u}\assign O\left(d\right)$,
$\fibreavg f{\omega}\left(z,n\right)\assign\int_{\mJ}fd\mmc$.

\subsection{Curvature}

We will approximate $\mJ$ with its $\ze$-\textit{parallel set} (Minkowski
sausage, dilation, offset, homogenization) $\mJ_{\ze}$ or the closure
of the complement: \begin{eqnarray}
\mJ_{\ze}\assign\left\{ x\in\Rd\,:\,\left|x-y\right|\leq\ze\text{ for a }y\in\mJ\right\} , &  & \widetilde{\mJ_{\ze}}\assign\overline{\left(\mJ_{\ze}\right)^{c}}.\label{eq:DEF parallel set}\end{eqnarray}
The \textit{reach} of a closed set $K\subseteq\Rd$ is the supremum
of all $s>0$ such that every point $x\in K_{s}$ has a unique nearest
neighbor $y\in K$. If $\op{reach}K>0$, the set of pairs $\left(y,\frac{x-y}{\left|x-y\right|}\right)$
forms the Federer \textit{normal bundle}, $\op{nor}K$ \cite{Federer59MR0110078}.
Positive reach replaces the $C^{\infty}$ smoothness often assumed
in differential geometry.

We will implicitly extend any conformal or similarity map $\xi$ of
$\Rd$ to one of $\Rd\times S^{d-1}$, $\xi\left(z,n\right)\assign\left(\xi\left(z\right),\xi^{\prime}\left(z\right)^{\op{orth}}n\right)$
(i.e., to the cotangent bundle equipped with the Sasaki metric). Here,
$\xi^{\prime\orth}$ is the \textit{orthogonal group component} of
its derivative $\xi^{\prime}$.\comment{ Write $\left(\xi^{\prime}\right)^{\orth-1}\assign\left(\left(\xi^{\prime}\right)^{\orth}\right)^{-1}$.}

We will need only these properties of \textit{Lipschitz-Killing curvature-direction
measures} $C_{k}$: Let $K\subseteq\Rd$ be a set of positive reach.
For $k\in\left\{ 0,\dots,d-1\right\} $, $C_{k}\left(K,\cdot\right)$
is a signed Borel measure on $\op{nor}K\subseteq\Rd\times S^{d-1}$
with (locally) finite variation. It is Euclidean motion \textit{covariant}
and \textit{homogeneous} of degree $k$, i.e., if $g$ is a similarity
with ratio $s>0$ and orthogonal group part $g^{\orth}$, then for
Borel $A\subseteq\Rd$, $N\subseteq S^{d-1}$, \begin{equation}
C_{k}\left(gK,gA\times g^{\orth}N\right)=s^{k}C_{k}\left(K,A\times N\right).\label{eq:Ck scaling}\end{equation}
 It is \textit{locally determined}, i.e., if two sets of positive
reach $K$, $\hat{K}$ agree $K\cap G=\hat{K}\cap G$ on an open set
$G\subseteq\Rd$, then for $A\subseteq\Rd$, $N\subseteq S^{d-1}$, 

\begin{equation}
C_{k}\left(K,\left(A\cap G\right)\times N\right)=C_{k}\left(\hat{K},\left(A\cap G\right)\times N\right).\label{eq:Ck locality}\end{equation}
A special case of \textit{continuity} is stated in Fact \ref{fac:continuity of curvature}
below. The $\Rd$ projection of $C_{d-1}\left(K,\cdot\right)$ agrees
with half the $\left(d-1\right)$-dimensional Hausdorff measure $\mathcal{H}^{d-1}$
\textit{(surface area) }on the boundary $\partial K$. Because the
\textit{Lebesgue} measure $\mathcal{L}^{d}$ and the rotation invariant
probability $\mathcal{H}_{S^{d-1}}$ on the sphere share these properties,
we define for Borel $A\subseteq\Rd$, $N\subseteq S^{d-1}$, \[
C_{d}\left(K,A\times N\right)\assign\mathcal{L}^{d}\left(K\cap A\right)\mathcal{H}_{S^{d-1}}\left(N\right).\]
(The trivial product with $\mathcal{H}_{S^{d-1}}$ makes every $C_{k}$
live on $\Rd\times S^{d-1}$.) We define the curvature measure of
$\mJ_{\ze}$ via the normal reflection since we assume $0<\op{reach}\widetilde{\mJ_{\ze}}$
(Assumption \ref{ass:Regularity-of-parallel-sets} below): for Borel
$A\subseteq\Rd$, $N\subseteq S^{d-1}$, $k<d$, \begin{equation}
C_{k}\left(\mJ_{\ze},A\times N\right)\assign\left(-1\right)^{d-1-k}C_{k}\left(\widetilde{\mJ_{\ze}},A\times\left\{ -n:n\in N\right\} \right).\label{eq:Ck mirrored}\end{equation}
This is consistent in case both $\mJ_{\ze}$ and $\widetilde{\mJ_{\ze}}$
are (unions of) sets of positive reach (\cite[Theorems 3.3, 3.2]{RatajZaehle01MR1846894}).
Of course, $C_{k}\left(\mJ_{\ze},\cdot\right)$ enjoys the above geometrical
properties. All assertions in this paper also hold for (non-directional)
\textit{Federer curvature measures}, i.e., their projection $C_{k}\left(\mJ_{\ze},\cdot\times S^{d-1}\right)$
onto $\Rd$. (We conjecture projecting to $\Rd$ does not reduce the
variation's mass \cite[Rem. 3.15]{BohlZaehle12CurvatureDirectionMeasures1111.4457}.)

For a gentle introduction to curvature, combine the brief summary
in \cite{Zaehle11SelfsimRandomFractalsMR2763731} with \cite{SchneiderWeil92IntegralgeometrieMR1203777}
(polyconvex case), \cite{KrantzParks08GeometricIntegrationTheoryMR2427002}
(currents and Hausdorff measure), \cite{Federer59MR0110078} (positive
reach), \cite{Zaehle86IntRepMR849863} (curvature-direction measure),
\cite{Bernig03AspectsCurvature} (survey of notions).

To put curvature in context: The total masses $K\mapsto C_{k}\left(K,\Rd\times S^{d-1}\right)$
form a complete system of Euclidean invariants in the following sense.
Every set-additive, continuous, motion invariant functional on the
space of convex bodies is a linear combination of total curvatures
(Hadwiger's theorem). By means of an approximation argument, this
holds for large classes of singular sets, including sets of positive
reach \cite{Zaehle90CharacterizationCurvatureMR1089237}. Furthermore,
$C_{k}$ is the integral of symmetric functions of generalized principal
curvatures $\kappa_{i}$ over the $d-1$-dimensional Hausdorff measure
$\mathcal{H}^{d-1}$ on the normal bundle: \[
C_{k}\left(K,A\times N\right)=c_{k,d}\int\limits _{\left(\op{nor}K\right)\cap A\times N}\frac{\sum_{i_{1}<\dots<i_{d-1-k}}\kappa_{i_{1}}\left(x,n\right)\dots\kappa_{i_{d-1-k}}\left(x,n\right)}{\prod_{j=1}^{d-1}\sqrt{1+\kappa_{j}^{2}\left(x,n\right)}}d\mathcal{H}^{d-1}\left(x,n\right).\]

\subsection{\label{sub:Formulas-(dumping-ground)}Remaining notation}

\begin{figure}
\printornot{\MpOrRaster{\includegraphics[scale=3]{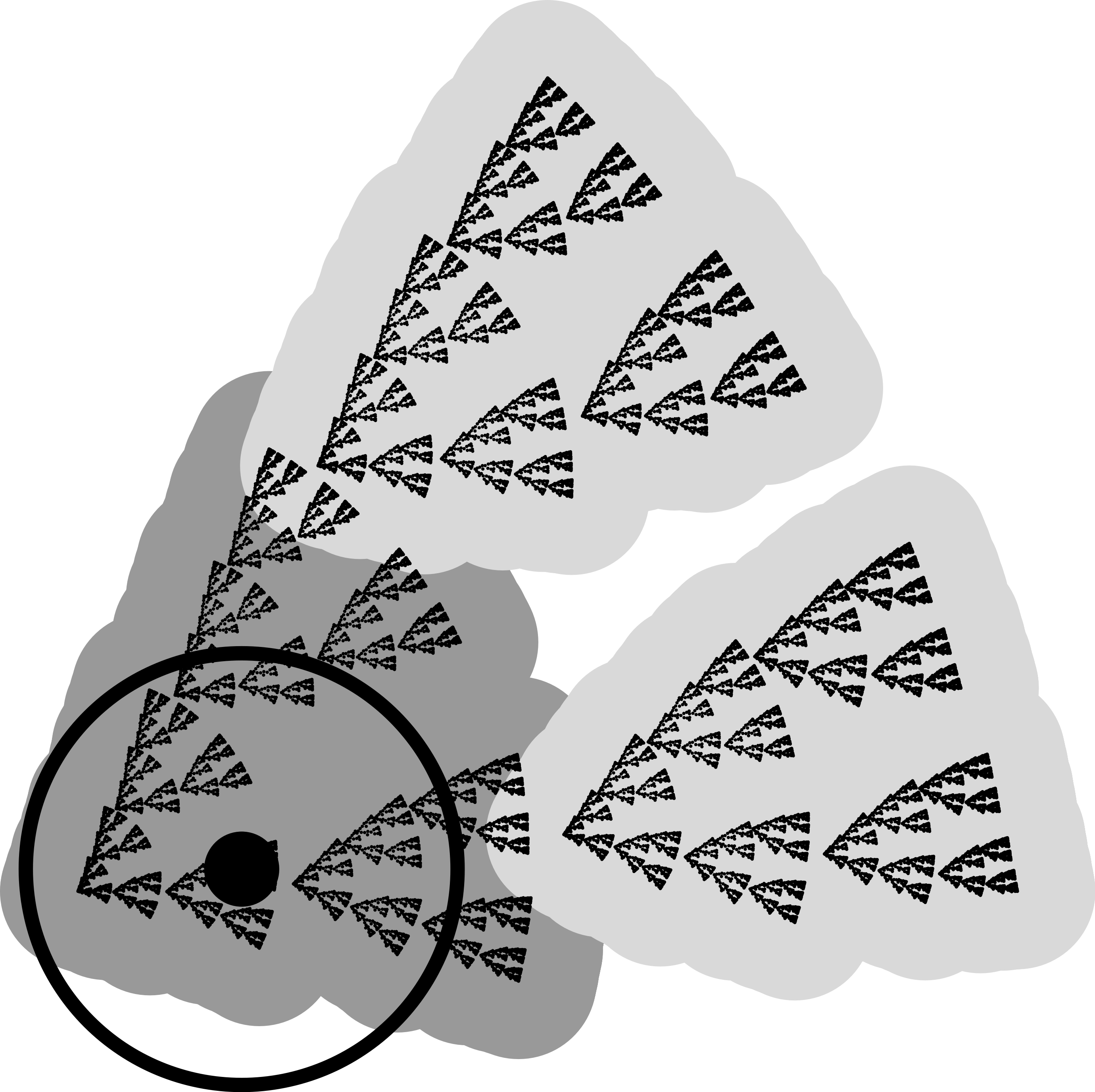}}{\includegraphics[scale=3]{psi_orig}}\ \MpOrRaster{\includegraphics[scale=3]{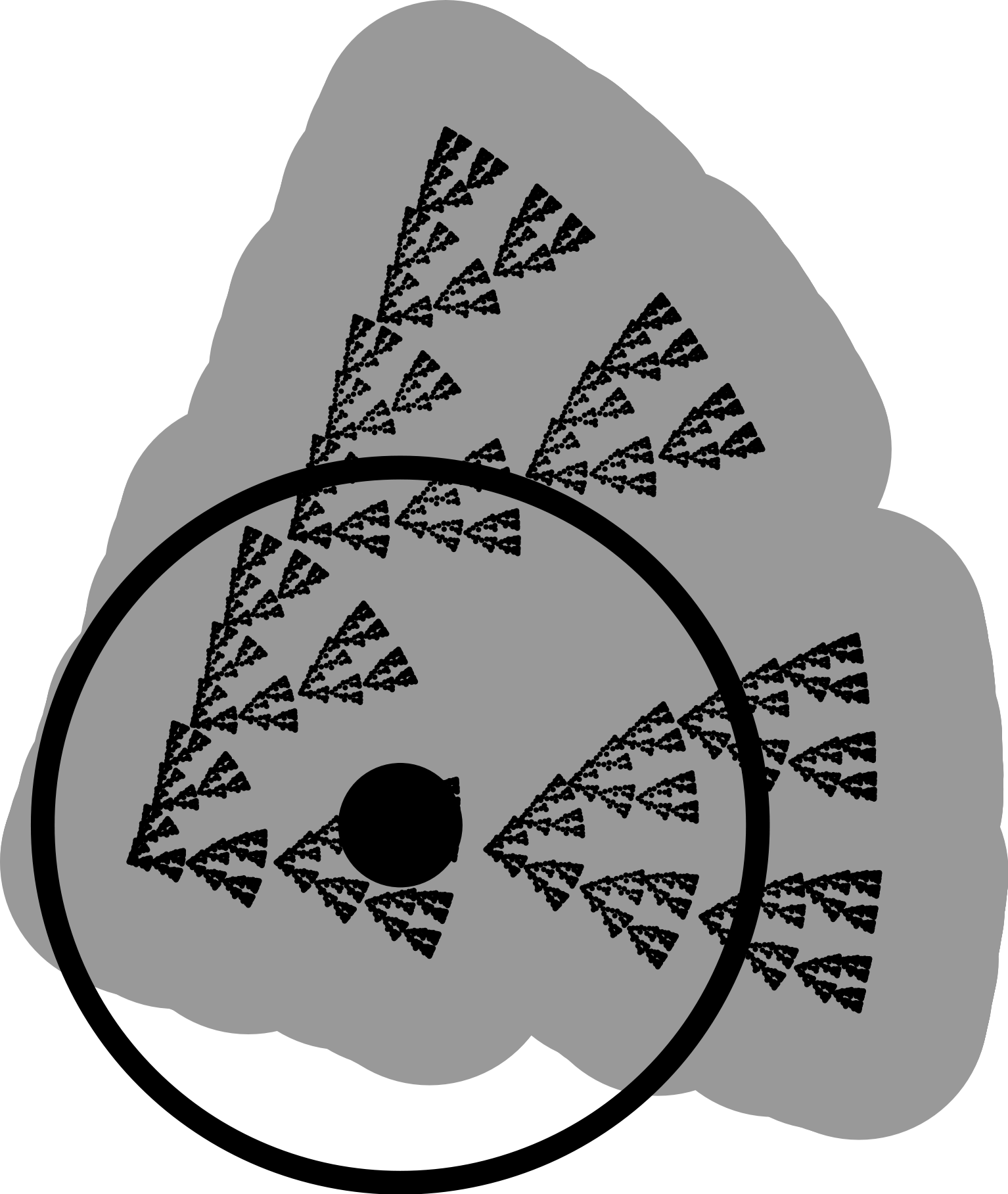}}{\includegraphics[scale=3]{psi_levelset}}\hspace*{1.3cm}\MpOrRaster{\includegraphics[scale=2.7]{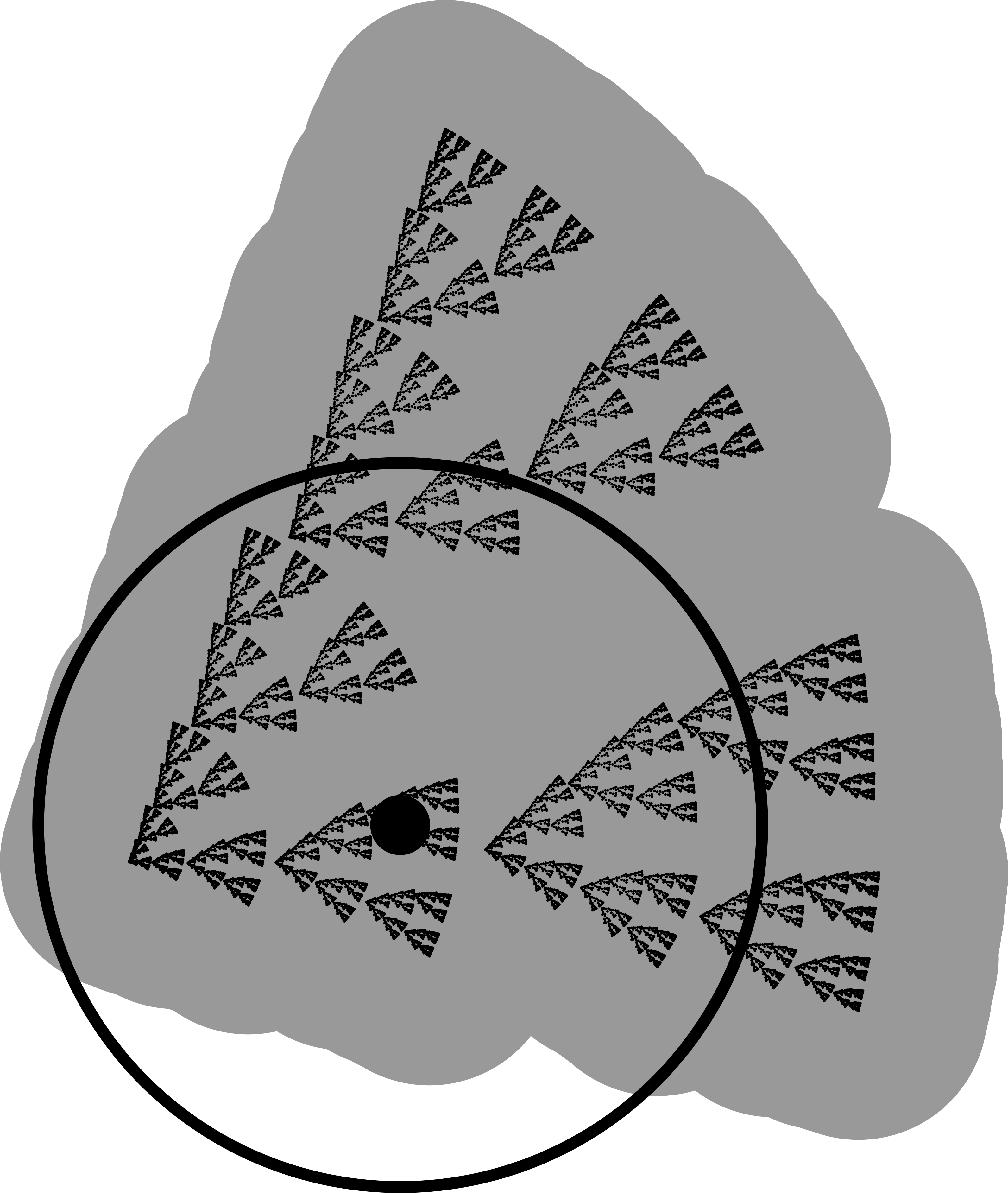}}{\includegraphics[scale=2.7]{psi_magnified}}}\caption{\label{fig:distortion and renewal radii}Left: fractal $\mJ$ (black),
parallel set $\mJ_{\ze}$ (both shades), point $x=\left(1,1,2,3,1,\dots\right)$
(marked) surrounded by a ball of $2a$ times its renewal radius $\R\left(x,1\right)$.
The ball supports $\Along$ and localizes curvature, so it must avoid
$\left(\mf 2F\right)_{\ze}\cup\left(\mf 2F\right)_{\ze}$ (light gray).\protect \\
Middle: level set $\left(\mf 1F\right)_{\ze}$ containing $x$.\protect \\
Right: Magnification is the composition of the dynamical system
map $\ms$ and distortion by $\bpsym$. Black/shaded: magnified $\left(\mf 1F\right)_{\ze}$
or $\left(\bp 1x\mJ\right)_{\ze/\zr 1x}$. The marked point $\bp 1x\ms x$
is surrounded with $2a\R\left(\ms x,0\right)$.}

\end{figure}
Fix arbitrary constants $a>1$ and $0<\Rstart\leq a^{-1}\dist\left(\mJ,\bV^{c}\right)$.
Let $\left(\eta,u\right)\in\mIN\times\mJ$ be any fixed reference
point. Define for $x,z\in\bV$, $\ze>0$,

\begin{equation}
\A xz{\ze}\assign\max\left(1-\frac{\left|x-z\right|}{\ze a},\,0\right)\label{eq:DEF A balls}\end{equation}
 a {}``smoothed-out'' indicator function of the $a\ze$-ball around
$x$. Inspired by $\R\left(x,m\right)\sim\dist\left(x,\partial\bigcup_{\tau\in\mI^{m}}\mf{\tau}\mX\right)$,
the \textit{renewal epochs} are defined (Definition \ref{def:R, N})
for $\mmc$-almost all $x\in\mJ$ and $m\in\mathbb{N}$ as:\begin{eqnarray}
\widetilde{\R}\left(x,0\right) & \assign & \frac{\dist\left(x,X^{c}\right)}{2\Kf a}\min\left\{ 1,\frac{2\Kf^{-1}a\Rstart}{\max_{x\in\mJ}\dist\left(x,\mX^{c}\right)}\right\} ,\nonumber \\
\R\left(x,m\right) & \assign & \sup_{n\geq m}\zr{x|n}{\ms^{n}x}\widetilde{\R}\left(\ms^{n}x,0\right).\label{eq:predef renewal radius R}\end{eqnarray}
 Denote the measure theoretic \textit{entropy} of $\mmi$ and $\ms$,
\begin{equation}
\entropy\assign-\zD\int_{\mJ}\ln\zr{x_{1}}{\ms x}\, d\mmi\left(x\right).\label{eq:DEF H is entropy}\end{equation}
Define for $x\in\mJ$, $\ze,s,t\in\left(0,\Rstart\right]$, $f:\bV\times S^{d-1}\ra\mathbb{R}$
Borel measurable, with the shorthand $\ex z=\left(z,n\right)\in\Rd\times S^{d-1}$,
the measure \begin{equation}
\hr_{\chunkh st}^{k,\pm}\left(x,\ze,f\right)\assign\int_{\op{nor}\widetilde{\mJ_{\ze}}}\frac{\chunk st\left(\ze\right)\,\ze^{\zD-k}\,\A xz{\ze}}{\int\A yz{\ze}d\mmc\left(y\right)}f\left(\ex z\right)\, dC_{k}^{\pm}\left(\mJ_{\ze},\ex z\right),\label{eq:Def hr nontransformed}\end{equation}
and for a finite word $\tau\in\mI^{*}$ in \eqref{eq:def psi} and
an infinite word $\omega\in\mIN$ in \eqref{eq: def psi limit}:\begin{multline}
\hw_{\tau}^{k,\pm}\left(x,\ze,f\right)\assign\\
\int\frac{\chunk{\R\left(x,1\right)}{\R\left(x,0\right)}\left(\frac{\ze}{\left|\bp{\tau}u^{\prime}x\right|}\right)\,\ze^{\zD-k}\:\A{\bp{\tau}ux}{\bp{\tau}uz}{\ze}}{\int\left|\bp{\tau}u^{\prime}y\right|^{\zD}\A{\bp{\tau}uy}{\bp{\tau}uz}{\ze}\, d\mmc\left(y\right)}f\left(\ex z\right)\, dC_{k}^{\pm}\left(\left(\bp{\tau}u\mJ\right)_{\ze},\bp{\tau}u\ex z\right),\label{eq:DEF hw finite-word}\end{multline}
\begin{multline}
\hw_{\widetilde{\omega}}^{k}\left(x,\ze,f\right)\assign\\
\int\frac{\chunk{\R\left(x,1\right)}{\R\left(x,0\right)}\left(\frac{\ze}{\left|\bplim{\omega}u^{\prime}x\right|}\right)\,\ze^{\zD-k}\:\A{\bplim{\omega}ux}{\bplim{\omega}uz}{\ze}}{\int\left|\bplim{\omega}u^{\prime}y\right|^{\zD}\A{\bplim{\omega}uy}{\bplim{\omega}uz}{\ze}\, d\mmc\left(y\right)}f\left(\ex z\right)\, dC_{k}\left(\left(\bplim{\omega}u\mJ\right)_{\ze},\bplim{\omega}u\ex z\right)\label{eq:DEF hw limiting word}\end{multline}
 if these integrals exist, and\begin{eqnarray*}
\hr_{\chunkh st}^{k}\assign\hr_{\chunkh st}^{k,+}-\hr_{\chunkh st}^{k,-}, &  & \hw_{\tau}^{k}\assign\hw_{\tau}^{k,+}-\hw_{\tau}^{k,-},\\
\hr_{\chunkh st}^{k,\var}\assign\hr_{\chunkh st}^{k,+}+\hr_{\chunkh st}^{k,-}, &  & \hw_{\tau}^{k,\var}\assign\hw_{\tau}^{k,+}+\hw_{\tau}^{k,-}.\end{eqnarray*}
The above transform into each other, see \eqref{eq:hw transforms-to hi chunk-of-curvdens}.
The proofs define further symbols: \comment{ $\mf i$ $\rmax$ $\bV$
$\mX$ $\mJ$ Assumption \ref{ass:model}, $\mmc$ $\zD$ \eqref{eq:perron frobenius, dimension, measure},
$\mmi$ $\mIN\times\mJ$ $\msb$ $\mmbc$ $\mmbi$ $\icdens$ next,
$\bp{\tau}x$ for finite $\tau\in\mI^{*}$ \eqref{eq:def psi}, $\bplim{\omega}x$
for $\omega\in\mIN$ \eqref{eq: def psi limit}, $\mJ_{\ze}$ $\widetilde{\mJ_{\ze}}$
\eqref{eq:DEF parallel set}, $C_{k}\left(\mJ_{\ze},\cdot\right)$
\eqref{eq:Ck mirrored}, $a$ $\Rstart$ $\left(\eta,u\right)$ Section
\ref{sub:Formulas-(dumping-ground)}, $\Along$ \eqref{eq:DEF A balls},
$\R\left(\cdot,\cdot\right)$ \eqref{eq:predef renewal radius R},
$\entropy$ \eqref{eq:DEF H is entropy}, $\hr$ \eqref{eq:Def hr nontransformed},
$\hw$ \eqref{eq:DEF hw finite-word} and \eqref{eq:DEF hw limiting word},
$\ex{\msb}$ \eqref{eq:DEF group extended shift}, $\G_{\eta u}$
$\HG$ \eqref{eq: DEF Brin group}, $\fibreavg f{\omega}$ \eqref{eq:DEF f limit},
$C_{k}^{\text{frac}}\left(\mJ,\cdot\right)$ \eqref{eq:limit of curvature},
$C_{\mJ}$ \eqref{eq:D-set-property Ahlfors}, $\mK$ \eqref{eq:K bounded distortion},
}$\Kf$ \eqref{eq:BDP.3}, $c_{\bpsym}$ $\mK_{\bpsym}$ Proposition
\ref{pro:(Distortion-converges)}, $C_{\Along}$ \eqref{eq:Along -- constant CA},
$\hi$ \eqref{eq: DEF hi}.

\section{Main result and Limit formula}

Recall $\mJ=\bigcup_{i\in\mI}\mf i\mJ$ is self-conformal with (OSC),
see Assumption \ref{ass:model}. \newbox\pardistorted \setbox\pardistorted=\hbox{$\left(\bplim{\omega}u\mJ\right)_{\ze}$}
\begin{assumption}
\label{ass:Regularity-of-parallel-sets}(Regularity of parallel sets)
Throughout this entire paper, whenever $k<d$, we will assume the
following. Lebesgue almost all $\ze>0$ shall be regular values of
the Euclidean distance function to $\mJ$ and to $\bplim{\omega}u\mJ$
for $\mmbc$-almost every $\omega$.

Alternatively, let almost all $\ze>0$, $\omega\in\mIN$ satisfy
\begin{enumerate}
\item $\op{reach}\widetilde{\mJ_{\ze}}>0$, \label{enu:ass:Regularity-of-parallel-sets  1 reach}
\item if $\left(y,m\right)\in\op{nor}\widetilde{\mJ_{\ze}}$, then $\left(y,-m\right)\notin\op{nor}\widetilde{\mJ_{\ze}}$,\label{enu:ass:Regularity-of-parallel-sets 2 osculating}
\item $\op{reach}\widetilde{\copy\pardistorted}>0$. \label{enu:ass:Regularity-of-parallel-sets 3 reach distorted}
\end{enumerate}
If the surface area measure $C_{d-1}$ is restricted to $\Rd$ instead
of $\Rd\times S^{d-1}$, the assumption is needed only for $k<d-1$.
\end{assumption}
This is always satisfied in ambient dimensions $d\leq3$ due to \cite{Fu85TubularNeighMR816398}.
It is also true for any strictly self-similar set whose convex hull
is a polytope (\cite{Pokorny11CriticalValuesSelfSimilar1101.1219}).
Regular distance values imply the alternative conditions, see \cite[Theorem 4.1]{Fu85TubularNeighMR816398},
\cite[Proposition 3]{RatajZaehle03NormalCyclesByApproxMR1983898}.
\begin{assumption}
\label{ass:(Uniform-integrability)}(Uniform integrability) Assume
at least one of the following:
\begin{enumerate}
\item $k=d$ (Minkowski content),
\item $k=d-1$ (Surface area content),
\item \[
\esup\limits _{\ze>0,x\in\mJ}\ze^{-k}C_{k}^{\var}\left(\mJ_{\ze},\Int B\left(x,a\ze\right)\right)<\infty,\]
 $\mmc$-essential in $x$, Lebesgue in $\ze$.\label{(regularity ass) enu: essential sup}
\item (Zygmund space) \[
h\left(\ze,\omega,x\right)\assign\ze^{-k}\sup_{M>0}\frac{1}{M}\sum_{m=0}^{M-1}C_{k}^{\var}\left(\left(\bp{\widetilde{\omega\vert m}}u\mJ\right)_{\ze},\bp{\widetilde{\omega\vert m}}u\Int B\left(x,a\ze\right)\right)\]
 satisfies both \begin{eqnarray*}
\int_{\mIN\times\mJ}\int_{\R\left(x,1\right)}^{\R\left(x,0\right)} & \max\left\{ 0,h\left(\ze,\omega,x\right)\,\ln h\left(\ze,\omega,x\right)\right\}  & \frac{d\ze}{\ze}d\mmbi\left(\omega,x\right)<\infty,\\
\int_{\mIN\times\mJ}\int_{0}^{\Rstart} & \sup_{0<\Rvar\leq\frac{\Rstart}{2}}\frac{1_{\chunkh{\max\left\{ \R\left(x,0\right),\Rvar\right\} }{\Rstart}}\left(\ze\right)C_{k}^{\var}\left(\mJ_{\ze},\Int B\left(x,a\ze\right)\right)}{\ln\left(\Rstart/\Rvar\right)\,\ze^{k}} & \frac{d\ze}{\ze}d\mmbi\left(\omega,x\right)<\infty.\end{eqnarray*}

\item As above, except integrating $h\ln h$ from $\R\left(x,1\right)/\mK$
to $\R\left(x,0\right)\mK$ and with $h\left(\ze,\omega,x\right)\assign$
\[
\ze^{-k}\sup_{M>0}\frac{1}{M}\sum_{m=0}^{M-1}C_{k}^{\var}\left(\left(\bp{\widetilde{\omega\vert m}}u\mJ\right)_{\ze},\bp{\widetilde{\omega\vert m}}u\Int B\left(x,a\ze\right)\right)\,\chunk{\R\left(x,1\right)}{\R\left(x,0\right)}\left(\frac{\ze}{\zrp{\widetilde{\omega\vert m}}ux}\right).\]
 
\item The family of functions\begin{equation}
\left(\omega,x,\ze\right)\mapsto\frac{1}{\ln\frac{\Rstart}{\Rvar}}\left[\hr_{\chunkh{\max\left\{ \R\left(x,0\right),\Rvar\right\} }{\Rstart}}^{k,\var}\left(x,\ze,1\right)+\sum_{m\in\mathbb{N}}1\left\{ \ze\zr{\widetilde{\omega\vert m}}u>\Rvar\right\} \hw_{\widetilde{\omega\vert m}}^{k,\var}\left(x,\ze,1\right)\right]\label{eq:uniformly integrand Thm}\end{equation}
is uniformly $\ze^{-1}d\ze\, d\mmbi\left(\omega,x\right)$-integrable
for all parameters $\Rvar\in\left(0,\Rstart/2\right)$.
\item As above, except the integration variable $x$ replaces $u$, including
inside $\hw_{\widetilde{\omega\vert m}}^{k,\var}$.
\end{enumerate}
\end{assumption}
Intuitively, the assumption limits how much curvature the overlap
sets may carry. More precisely, let $\tau$, $\vartheta\in\mI^{*}$
be different shortest words such that the diameters of $\bp{\widetilde{\omega|m}}u\mf{\tau}\mJ$
and $\bp{\widetilde{\omega|m}}u\mf{\vartheta}\mJ$ are smaller than
$\ze a^{-1}/\sqrt{2}$. Then to check (uniform) integrability of its
$\hw_{\widetilde{\omega\vert m}}^{k,\var}$- or $\ze^{-k}C_{k}^{\var}\left(\left(\bp{\widetilde{\omega|m}}u\mJ\right)_{\ze},\cdot\right)$-mass,
only the intersection of $\bp{\widetilde{\omega\vert m}}u\Int B\left(x,a\ze\right)$
with the union of overlap sets $\left(\bp{\widetilde{\omega|m}}u\mf{\tau}\mJ\right)_{\ze}\cap\left(\bp{\widetilde{\omega|m}}u\mf{\vartheta}\mJ\right)_{\ze}$
of all such pairs $\tau$, $\vartheta$ matters. The curvature on
the non-overlap portion of $\bp{\widetilde{\omega\vert m}}u\Int B\left(x,a\ze\right)$
satisfies item \eqref{(regularity ass) enu: essential sup} above,
see \cite[Theorem 4.1]{Zaehle11SelfsimRandomFractalsMR2763731}. 

In case of self-similar fractals, further sufficient conditions are
the Strong Curvature Bound Condition \cite{Winter10CurvBoundsArXiv10102032}
and polyconvex parallel sets \cite[Lemma 5.3.2]{Win08MR2423952}.
Their advantage is they do not involve parallel sets widths $\ze$
of all orders. (In our Zygmund space condition, $x\mapsto\R\left(x,1\right)$
is bounded from below if the Strong Separation Condition holds.)

Not all self-conformal sets satisfy item \eqref{(regularity ass) enu: essential sup}
(see \cite{RatajZaehle10CurvatureDensitiesSelfSimilarSets}), hence
the more general conditions.

Recall $\hw$ \eqref{eq:DEF hw limiting word} and  $\fibreavg f{\omega}$
\eqref{eq:DEF f limit}.
\begin{thm}
\label{thm:Main result}Let $k\in\left\{ 0,\dots,d\right\} $, and
$\mJ$ be a self-conformal set (Assumption \ref{ass:model}). Suppose
Assumptions \ref{ass:Regularity-of-parallel-sets} (only if $k\leq d-1$)
and \ref{ass:(Uniform-integrability)} hold. Then for any continuous
$f:\bV\times S^{d-1}\ra\mathbb{R}$, the following limit exists and
is finite: \begin{eqnarray}
C_{k}^{\text{frac}}\left(\mJ,f\right) & \assign & \lim_{\Rvar\searrow0}\frac{1}{\ln\Rstart/\Rvar}\int\limits _{\Rvar}^{\Rstart}f\left(\ex z\right)\ze^{\zD-k}dC_{k}\left(\mJ_{\ze},\ex z\right)\frac{d\ze}{\ze}\label{eq:limit of curvature}\\
 & = & \frac{\zD}{\entropy}\int\limits _{\mIN\times\mJ}\int\limits _{0}^{\infty}\hw_{\widetilde{\omega}}^{k}\left(x,\ze,\fibreavg f{\omega}\right)\,\frac{d\ze}{\ze}d\mmbi\left(\omega,x\right)\nonumber \end{eqnarray}
 \begin{multline*}
=\frac{\zD}{\entropy}\int\limits _{\mIN\times\mJ}\int\limits _{\R\left(x,1\right)}^{\R\left(x,0\right)}\int\limits _{\op{nor}\widetilde{\mJ_{\ze}}}\frac{\ze^{\zD-k}\:\Along\left(\left|\bplim{\omega}xz\right|,\ze\right)\,\fibreavg f{\omega}\left(\ex z\right)}{\int\left|\bplim{\omega}x^{\prime}y\right|^{\zD}\A{\bplim{\omega}xy}{\bplim{\omega}xz}{\ze}\, d\mmc\left(y\right)}\\
dC_{k}\left(\left(\bplim{\omega}x\mJ\right)_{\ze},\bplim{\omega}x\ex z\right)\,\frac{d\ze}{\ze}d\mmbi\left(\omega,x\right).\end{multline*}

\end{thm}
The first limit formula formally depends on $u$ via $\bplim{\omega}u$
inside $\hw_{\widetilde{\omega}}^{k}$. Its advantage is less objects
depend on the integration variable $x$. The second formula is better
adapted to the dynamical system variable $\left(\omega,x\right)$
and has $u$ replaced with $x$. We conjecture the uniform integrability
assumption \eqref{eq:uniformly integrand Thm} cannot be weakened,
see Remark \ref{rem:sharp}. The proof is postponed after Remark \ref{rem:sharp}.
Future work will simplify the formula for special cases.

For readers interested only in the Minkowski content, or in non-directional
fractal curvature measures on $\Rd$ instead of $\Rd\times S^{d-1}$:
\begin{cor}
(Non-directional version) Assume the situation of the theorem, except
that $f:\Rd\times S^{d-1}$ does not depend on its $S^{d-1}$ coordinate,
and positive reach (Assumption \ref{ass:Regularity-of-parallel-sets})
is only needed for $k\leq d-2$. Then the limit \eqref{eq:limit of curvature}
exists. That limit formula holds, and $C_{k}^{\text{frac}}\left(\mJ,\cdot\right)$
is proportional to $\mmc$, and $\fibreavg f{\omega}=\int fd\mmc$. 
\end{cor}
In the self-similar setting, previous work used various definitions
of $\Along$ (Remark \ref{rem:Our-A-unifies}, \cite[Example 2.1.1]{RatajZaehle10CurvatureDensitiesSelfSimilarSets}).
We unify those approaches.
\begin{cor}
Let every $\mf i$ be a similarity, in addition to Assumptions \ref{ass:model},
\ref{ass:Regularity-of-parallel-sets}, \ref{ass:(Uniform-integrability)}.
(Note $\bpsym_{\dots}=\id$, $\zrall{\bpsym_{\dots}}{}=1$ therein.)
Let $\Along$ be as above, or the indicator function of any neighborhood
net from \cite[Example 2.1.1]{RatajZaehle10CurvatureDensitiesSelfSimilarSets},
or defined any other way that makes Lemma \ref{lem:(neighborhood-net)}
true. Writing $\mathcal{H}^{\zD}$ for the $\zD$-dimensional Hausdorff
measure, $\mathcal{H}_{\G}$ for the Haar probability on $\G\assign\overline{\left<\mf i^{\prime}\,:\, i\in\mI\right>}_{O\left(d\right)}$,
and $c_{R}$ for the constants in \eqref{eq:predef renewal radius R},
we have \begin{alignat*}{2}
C_{k}^{\text{frac}}\left(\mJ,f\right) & = & \frac{\zD}{\entropy}\int\limits _{\mJ}\int\limits _{c_{R}\dist\left(x,\mf{x|1}X^{c}\right)}^{c_{R}\dist\left(x,X^{c}\right)}\int\limits _{\op{nor}\widetilde{\mJ_{\ze}}}\frac{\ze^{\zD-k}\:\Along\left(\left|z-x\right|,\ze\right)\,\fibreavg f{}\left(\ex z\right)}{\int_{\mJ}\A yz{\ze}\, d\mathcal{H}^{\zD}\left(y\right)}dC_{k}\left(\mJ_{\ze},\ex z\right)\,\frac{d\ze}{\ze}d\mathcal{H}^{\zD}\left(x\right),\\
\fibreavg f{}\left(z,n\right) & \assign & \mathcal{H}^{\zD}\left(\mJ\right)^{-1}\,\int_{\mJ}\int_{\G}f\left(x,gn\right)\, d\mathcal{H}^{\zD}\left(x\right)d\mathcal{H}_{\G}\left(g\right).\end{alignat*}
\end{cor}
\begin{rem}
\label{rem:Arguably renormalize Ck}Arguably, the fractal curvature-direction
measures should instead be normalized as \begin{eqnarray}
\hat{C}_{k}^{\op{frac}}\left(\mJ,\cdot\right) & \assign & \binom{d-\zD_{k}}{d-k}^{-1}\pi^{\frac{\zD_{k}-k}{2}}\frac{\Gamma\left(\frac{d-\zD_{k}}{2}+1\right)}{\Gamma\left(\frac{d-k}{2}+1\right)}C_{k}^{\op{frac}}\left(\mJ,\cdot\right),\label{eq:renormalize fractal curvature}\\
\zD_{k} & \assign & \min\left\{ k,\overline{\op{dim}}_{\text{Minkowski}}\mJ\right\} .\nonumber \end{eqnarray}
The motivation is, embedding a set $\myset$ into a larger ambient
dimension $d$ does not create geometric information. Stronger yet,
Zähle conjectured the fractal curvatures $k>\zD$ never provide new
information. This is proven in two cases: Fractal curvatures of an
order $k$ greater than the dimension of a {}``classical'' set $\myset$
simply repeat the highest available order of curvature (due to the
Steiner formula), \[
\hat{C}_{k}^{\op{frac}}\left(\myset,\cdot\right)=C_{\min\left\{ k,\op{dim}\myset\right\} }\left(\myset,\cdot\right)\,\text{ if }\op{reach}\myset>0.\]
 For a general, closed, Lebesgue null set $\myset$, the fractal Lebesgue
measure repeats the next-lower measure, \cite{RatajWinter09MeasuresOfParallelSetsArxiv09053279,RatajWinter11MinkowskiMeasurabilitySurface1111.1825}:
\begin{equation}
\hat{C}_{d}^{\op{frac}}\left(\myset,\cdot\times S^{d-1}\right)=\hat{C}_{d-1}^{\op{frac}}\left(\myset,\cdot\times S^{d-1}\right)\,\text{ if }C_{d}\left(\myset\right)=0.\label{eq:Cd_Cd-1}\end{equation}

\end{rem}

\section{Proofs}

\subsection{\label{sub:Exposing-the-dynamical}Exposing the dynamical system}

The purpose of the next proposition is to rephrase the problem, of
taking Cesaro limits of volume or curvature of parallel sets, into
the language of the ergodic shift dynamical system $\left(\mJ,\ms,\mmi\right)$.
The geometric intuition is to study the curvature in a small ball
around $x\in\mJ$ as the parallel set width $\ze$ decreases. The
conformal measure $\mmc$ of this ball balances out the rescaling
factor. Instead of actual balls, we will use a tent function $\Along$
supported by the ball \eqref{eq:DEF A balls}. We introduce the extra
integrand $1=\int Ad\mmc/\int Ad\mmc$ into the curvature and use
Fubini to make the {}``dynamical'' measure $\mmc$ accessible. The
analytic intuition is to convolute the curvature measure with an integrable
bump function $\Along$. If we had a sensible (harmonic analysis)
group structure on the fractal, it would preserve the metric. So we
make the bump $x\mapsto\A xz{\ze}$ depend only on pairwise distances
$\left|x-z\right|$. Lacking a global group, we normalize the bump
mass locally at each midpoint $x$.

This leads us to study the dynamical system's integrands $\hr$, $\hw$
for the rest of this paper.

Recall from \eqref{eq:Def hr nontransformed}, \[
\hr_{\chunkh st}^{k,\pm}\left(x,\ze,f\right)\assign\chunk st\left(\ze\right)\int_{\op{nor}\widetilde{\mJ_{\ze}}}f\left(\ex z\right)\frac{\ze^{\zD-k}\,\A xz{\ze}}{\int\A yz{\ze}d\mmc\left(y\right)}\, dC_{k}^{\pm}\left(\mJ_{\ze},\ex z\right).\]

\begin{rem}
\label{rem:measurability}Borel measurability of $\hr$ and $\hw$
as a function of $\ze$ can be proved as in \cite[Lemma 2.3.1]{RatajZaehle10CurvatureDensitiesSelfSimilarSets}.
The mapping $\ze\mapsto C_{k}\left(\mJ_{\ze},\cdot\right)$ is weak{*}-continuous
at almost every $\ze$ (see \cite[Corollary 2.3.5]{Zaehle11SelfsimRandomFractalsMR2763731}).
The non-osculating Assumption \ref{ass:Regularity-of-parallel-sets}
\eqref{enu:ass:Regularity-of-parallel-sets 2 osculating} is used
here. Since $C_{k}\left(\left(\bpsym\mJ\right)_{\ze},\cdot\right)$
agrees with a (limit of) pullback of some $C_{k}\left(\left(\bpsym\mJ\right)_{\text{const }\ze},\cdot\right)$
on sets we integrate over, we will see it is automatically measurable.\end{rem}
\begin{prop}
\label{pro:Fubini}Let $k\in\left\{ 0,\dots,d\right\} $, $\Rvar>0$,
and $f:\bV\times S^{d-1}\ra\left[0,\infty\right)$ be any nonnegative,
measurable function. Suppose Assumption \ref{ass:model} governs our
self-conformal set $F$. Its parallel sets $\mJ_{\ze}$ shall be regular
if $k\leq d-1$ (but only $k\leq d-2$ for $f:\bV\ra\left[0,\infty\right)$),
Assumption \ref{ass:Regularity-of-parallel-sets}. Then we have 

\begin{equation}
\frac{1}{\ln\Rstart/\Rvar}\int_{\Rvar}^{\Rstart}f\left(\ex z\right)\,\ze^{\zD-k}dC_{k}^{\pm}\left(\mJ_{\ze},\ex z\right)\frac{d\ze}{\ze}=\int_{\mJ}\int_{0}^{\infty}\frac{1}{\ln\Rstart/\Rvar}\,\hr_{\chunkh{\Rvar}{\Rstart}}^{\left(k,\pm\right)}\left(x,\ze,f\right)\,\frac{d\ze}{\ze}d\mmc\left(x\right).\label{eq:integral-of-curvature-density}\end{equation}
\end{prop}
\begin{proof}
At each point $z\in\mJ_{\ze}$ we may introduce an arbitrary factor
\begin{equation}
1=\frac{\int\A xz{\ze}d\mmc\left(x\right)}{\int\A yz{\ze}d\mmc\left(y\right)}.\label{eq:A cancels}\end{equation}
Lemma \ref{lem:(neighborhood-net)} \eqref{enu:lem:(neighborhood-net) Measurability},
\eqref{enu:lem:(neighborhood-net) D-set-property} makes sure $\int\Along\, d\mmc>0$
and $\Along\geq0$. Next, we may apply Fubini's theorem because we
are dealing with positive integrands and finite measures. If $\Rvar<\ze\leq\Rstart$:
\begin{multline*}
\ze^{\zD-k}\int_{\op{nor}\widetilde{\mJ_{\ze}}}f\left(\ex z\right)\, dC_{k}^{\pm}\left(\mJ_{\ze},\ex z\right)=\ze^{\zD-k}\int_{\op{nor}\widetilde{\mJ_{\ze}}}f\left(\ex z\right)\,\frac{\int\A xz{\ze}d\mmc\left(x\right)}{\int\A yz{\ze}d\mmc\left(y\right)}\, dC_{k}^{\pm}\left(\mJ_{\ze},\ex z\right)\\
=\int_{J}\int_{\op{nor}\widetilde{\mJ_{\ze}}}f\left(\ex z\right)\frac{\ze^{\zD-k}\,\A xz{\ze}}{\int\A yz{\ze}d\mmc\left(y\right)}\, dC_{k}^{\pm}\left(\mJ_{\ze},\ex z\right)d\mmc\left(x\right)=\int_{J}\hr_{\chunkh{\Rvar}{\Rstart}}^{\left(k,\pm\right)}\left(x,\ze,f\right)d\mmc\left(x\right).\end{multline*}
  Finally, we integrate over $\ze^{-1}d\ze$ and apply Fubini again.
\end{proof}

\subsection{\label{sub:Dualizing-the-dynamics}Dualizing the dynamics to apply
the Birkhoff theorem and leave the distortion behind}

Here we tell the story arc of this paper and treat the dynamics. The
curvature Cesaro average is split into chunks mapped onto each other
up to distortion of the underlying fractal. The local geometric covariance
matches up only their curvature measures. To compensate, the dynamical
system pulls back the test function. Asymptotics of the distortion
exist only if the ordering of the conformal maps $\mf i$ is reversed.
The Perron-Frobenius operator achieves this, but requires the two-sided
extension $\mIN\times\mJ$ of the code space. (Intuitively, summing
over all possible $n$-letter words is the same as summing over reversed
words.) It also separates the distortion inside the measure from the
dynamics map inside the test function. A vector (Kojima) Toeplitz
theorem moves the curvature measures out of the Cesaro average. The
Birkhoff theorem for the left shift fits the Cesaro-averaged, pulled-back
chunks of test function after compensating the following: The curvature
lives on a parallel set {}``beside'' the dynamical system $\mJ$.
This selects a different element from the covariance group element
that transforms the normal directions. 

This work essentially reduces the self-conformal case to similarities,
at the price of extra distortion. The distortion and the geometry
packaged in $\hr$, $\hw$ will later be treated point-wise using
the right shift $\msb^{-1}$. From a dynamics point of view, the different
roles of left and right shift explain why the conformal measure $\mmbc$
integrates the test function in the limiting $C_{k}^{\op{frac}}$.

Recall $\hr$ from \eqref{eq:Def hr nontransformed}, $\hw$ from
\eqref{eq:DEF hw finite-word} and \eqref{eq:DEF hw limiting word}.
\begin{prop}
\label{pro:(core of proof)}(core of proof) Assume $k\in\left\{ 0,\dots,d\right\} $,
$f:\bV\times S^{d-1}\ra\mathbb{R}$ continuous and $\mJ$ is self-conformal
(Assumption \ref{ass:model}). Let its parallel sets $\mJ_{\ze}$
be regular if $k<d$ (only $k<d-1$ if $f:\bV\ra\mathbb{R}$) (positive
reach Assumption \ref{ass:Regularity-of-parallel-sets}). Let the
expression \begin{equation}
\frac{1}{\ln\frac{\Rstart}{\Rvar}}\left[\hr_{\chunkh{\max\left\{ \R\left(x,0\right),\Rvar\right\} }{\Rstart}}^{k,\var}\left(x,\ze,1\right)+\sum_{n\in\mathbb{N}}1\left\{ \ze\zr{\widetilde{\omega\vert n}}u>\Rvar\right\} \hw_{\widetilde{\omega\vert n}}^{k,\var}\left(x,\ze,1\right)\right]\label{eq:uniformly integrand-1}\end{equation}
be $\Rvar$-uniformly ~$\ze^{-1}d\ze\, d\mmbc\left(\omega,x\right)$-integrable.
Then for any continuous $f:\bV\times S^{d-1}\ra\mathbb{R}$, the limit
\begin{eqnarray}
 &  & \lim_{\Rvar\searrow0}\frac{1}{\ln\Rstart/\Rvar}\int_{\Rvar}^{\Rstart}\ze^{\zD-k}\int_{\Rd\times S^{d-1}}f\left(\ex z\right)dC_{k}\left(\mJ_{\ze},\ex z\right)\frac{d\ze}{\ze}\label{eq:limit of curvature in core prop}\\
 & = & \frac{\zD}{\entropy}\int_{\mIN\times\mJ}\int_{0}^{\infty}\hw_{\widetilde{\omega}}^{k}\left(x,\ze,\fibreavg f{\omega}\right)\,\frac{d\ze}{\ze}d\mmbi\left(\omega,x\right)\nonumber \end{eqnarray}
 exists.\end{prop}
\begin{proof}
Due to linearity, we may assume $0\leq f\leq1$.

First, we will rewrite \eqref{eq:limit of curvature in core prop}
for any fixed $\Rvar$. To prevent that the $\ze^{-1}d\ze$-integral
evaluates to $\infty-\infty$, we will work with $C_{k}^{+}$ and
$C_{k}^{-}$ separately at first. Proposition \ref{pro:Fubini} below
uses Fubini's theorem to makes $\mmc$ appear in the next line, \[
L\assign\frac{1}{\ln\frac{\Rstart}{\Rvar}}\int_{\Rvar}^{\Rstart}\int_{\op{nor}\widetilde{\mJ_{\ze}}}f\left(\ex z\right)\,\ze^{\zD-k}dC_{k}^{\pm}\left(\mJ_{\ze},\ex z\right)\frac{d\ze}{\ze}=\frac{1}{\ln\frac{\Rstart}{\Rvar}}\int_{\mJ}\int_{0}^{\infty}\hr_{\chunkh{\Rvar}{\Rstart}}^{k,\pm}\left(x,\ze,f\right)\,\frac{d\ze}{\ze}d\mmc\left(x\right).\]
 The sequence sequence $n\mapsto\R\left(x,n\right)$ is strictly falling
(Lemma \ref{lem:(Renewal-radii)} \eqref{enu:lem:(Renewal-radii) 1 monotonicity}
\eqref{enu:lem:(Renewal-radii)  2 limit}). We can split the $d\ze$
integral along it without reversing the new bounds of any chunk of
the integral. The indicator function corresponding to the bounds $\int_{\R\left(x,n+1\right)}^{\R\left(x,n\right)}$
is pulled into the subscript of $\hr$, see \eqref{eq:Def hr nontransformed}:
\[
L=\int_{\mJ}\int_{0}^{\infty}\frac{\hr_{\chunkh{\max\left\{ \R\left(x,0\right),\Rvar\right\} }{\Rstart}}^{k,\pm}\left(x,\ze,f\right)}{\ln\frac{\Rstart}{\Rvar}}+\sum_{n\in\mathbb{N}}\frac{1\left\{ \ze>\Rvar\right\} }{\ln\frac{\Rstart}{\Rvar}}\hr_{\chunkh{\R\left(x,n+1\right)}{\R\left(x,n\right)}}^{k,\pm}\left(x,\ze,f\right)\,\frac{d\ze}{\ze}d\mmc\left(x\right).\]
 Next, Lemma \ref{lem:preimages of hr, hw curvdens} below asserts
$\hr_{\chunkh{\R\left(x,n+1\right)}{\R\left(x,n\right)}}^{k,\pm}\left(x,\ze,f\right)=\hw_{x|n}^{k,\pm}\left(\ms^{n}x,\ze\zr{x|n}u^{-1},f\circ\mf{x|n}\right)$.
The geometric meaning is we preimage it under $\mf{x|n}$ and use
the covariance and locality properties of curvature. Hence, 

\[
L=\int_{\mJ}\int_{0}^{\infty}\dots+\sum_{n\in\mathbb{N}}\frac{1\left\{ \ze>\Rvar\right\} }{\ln\frac{\Rstart}{\Rvar}}\hw_{x|n}^{k,\pm}\left(\ms^{n}x,\frac{\ze}{\zr{x\vert n}u},f\circ\mf{x|n}\right)\,\frac{d\ze}{\ze}d\mmc\left(x\right).\]
Temporarily, we pull the formally infinite sum out of both integrals.
Each summand integral exists due to positivity (or uniform integrability
once proved), \[
L=\int_{\mJ}\int_{0}^{\infty}\dots+\sum_{n\in\mathbb{N}}\int_{\mJ}\int_{0}^{\infty}\frac{1\left\{ \ze>\Rvar\right\} }{\ln\frac{\Rstart}{\Rvar}}\hw_{x|n}^{k,\pm}\left(\ms^{n}x,\frac{\ze}{\zr{x\vert n}u},f\circ\mf{x|n}\right)\,\frac{d\ze}{\ze}d\mmc\left(x\right).\]
By passing to the two-sided code space and to the invariant measure
$\mmbi=\icdens\mmbc$, we can apply the shift operator, $\left(\ms^{n}\omega,\mf{\widetilde{\omega|n}}x\right)=\msb^{-n}\left(\omega,x\right)$,
\eqref{eq:lem:(Perron-Frobenius-operator)}. It consistently replaces
$x$ with $\mf{\widetilde{\omega\vert n}}x$. But left and right shifts
are inverse operations. Thus it replaces the original $\ms^{n}x$
with $\ms^{n}\mf{\widetilde{\omega\vert n}}x=x$ and $x|n$ with $\left(\mf{\widetilde{\omega\vert n}}x\right)|n=\widetilde{\omega\vert n}$:\[
L=\int_{\mJ}\int_{0}^{\infty}\dots+\sum_{n\in\mathbb{N}}\int_{\mIN\times\mJ}\int_{0}^{\infty}\frac{1\left\{ \ze>\Rvar\right\} }{\icdens\left(\mf{\widetilde{\omega\vert n}}x\right)\ln\frac{\Rstart}{\Rvar}}\hw_{\widetilde{\omega\vert n}}^{k,\pm}\left(x,\frac{\ze}{\zr{\widetilde{\omega\vert n}}u},f\circ\mf{\widetilde{\omega\vert n}}\right)\,\frac{d\ze}{\ze}d\mmbi\left(\omega,x\right).\]
We substitute $\ze/\zr{\widetilde{\omega\vert n}}u\mapsto\ze$ in
the $d\ze$ integral and put the sum back inside, \[
L=\int_{\mIN\times\mJ}\int_{0}^{\infty}\left[\dots+\sum_{n\in\mathbb{N}}\frac{1\left\{ \zr{\widetilde{\omega\vert n}}u>\frac{\Rvar}{\ze}\right\} }{\icdens\left(\mf{\widetilde{\omega\vert n}}x\right)\ln\frac{\Rstart}{\Rvar}}\hw_{\widetilde{\omega\vert n}}^{k,\pm}\left(x,\ze,f\circ\mf{\widetilde{\omega\vert n}}\right)\right]\,\frac{d\ze}{\ze}d\mmbi\left(\omega,x\right).\]
Because we want to move the $\Rvar$-limit inside both integrals,
we must assume the last line is uniformly integrable. The bounds $\mK^{-1}\leq\icdens\leq\mK$
and $\left\Vert f\circ\mf{\widetilde{\omega\vert n}}\right\Vert \leq1$
simplify that to our stated assumption \eqref{eq:uniformly integrand Thm}.

Uniform integrability of $L$ implies integrability. So we can subtract
our chain of equations for $C_{k}^{+}$ and $C_{k}^{-}$, i.e., $\hw_{\widetilde{\omega\vert n}}^{k,+}-\hw_{\widetilde{\omega\vert n}}^{k,-}$.
Then we draw the limit $\Rvar\ra0$ into the integrals. We have proved,
assuming the limit inside the double integral exists:\begin{multline}
\lim_{\Rvar\ra0}\frac{1}{\ln\frac{\Rstart}{\Rvar}}\int_{\Rvar}^{\Rstart}\int_{\op{nor}\widetilde{\mJ_{\ze}}}f\left(\ex z\right)\,\ze^{\zD-k}dC_{k}\left(\mJ_{\ze},\ex z\right)\frac{d\ze}{\ze}\,=\,\int_{\mIN\times\mJ}\int_{0}^{\infty}\lim_{\Rvar\ra0}\left[\frac{\hr_{\chunkh{\max\left\{ \R\left(x,0\right),\Rvar\right\} }{\Rstart}}^{k,\pm}\left(x,\ze,f\right)}{\ln\frac{\Rstart}{\Rvar}\icdens\left(x\right)}\right.\\
\left.+\frac{1}{\ln\frac{\Rstart}{\Rvar}}\sum_{n\in\mathbb{N}}1\left\{ \zr{\widetilde{\omega\vert n}}u>\frac{\Rvar}{\ze}\right\} \hw_{\widetilde{\omega\vert n}}^{k}\left(x,\ze,\frac{f\circ\mf{\widetilde{\omega\vert n}}}{\icdens\left(\mf{\widetilde{\omega\vert n}}x\right)}\right)\right]\frac{d\ze}{\ze}d\mmbi\left(\omega,x\right).\label{eq:now limit inside integral}\end{multline}
 The rest of this proof will show the integrand limit \eqref{eq:now limit inside integral}
exists. Clearly, $\hr_{\dots}^{k,\pm}$ can be ignored, \[
\lim_{\Rvar\ra0}\left(\ln\frac{\Rstart}{\Rvar}\right)^{-1}\frac{\hr_{\chunkh{\max\left\{ \R\left(x,0\right),\Rvar\right\} }{\Rstart}}^{k}\left(x,\ze,f\right)}{\icdens\left(x\right)}=0.\]

The Markov time $N\left(\ze/\Rvar\right)\assign\max\left\{ n\in\mathbb{N}\,:\,\zr{\widetilde{\omega\vert n}}u>\frac{\Rvar}{\ze}\right\} $
will help reinterpret the sum as Cesaro average. The integrand of
\eqref{eq:now limit inside integral} becomes \[
\lim_{\Rvar\ra0}\left[0+\frac{N\left(\ze/\Rvar\right)}{\ln\Rstart/\Rvar}\frac{1}{N\left(\frac{\ze}{\Rvar}\right)}\sum_{n\leq N\left(\frac{\ze}{\Rvar}\right)}\hw_{\widetilde{\omega\vert n}}^{k}\left(x,\ze,\frac{f\circ\mf{\widetilde{\omega\vert n}}}{\icdens\left(\mf{\widetilde{\omega\vert n}}x\right)}\right)\right].\]
 Lemma \ref{lem:Birkhoff contraction ratio} below allows us to replace
the following subexpression with its limit, \[
\frac{N\left(\ze/\Rvar\right)}{\ln\Rstart/\Rvar}\,\underset{\Rvar\ra0}{\longrightarrow}\,\frac{\zD}{\entropy}.\]
 That eliminated the only $\Rvar$ outside $N\left(\ze/\Rvar\right)$.
Recall $\zr{\widetilde{\omega\vert n}}u\leq\rmax^{n}$, $\rmax<1$.
Thus $\Rvar\ra0$ implies $N\left(\ze/\Rvar\right)\ra\infty$. We
replace $\lim_{\Rvar\ra0}$ with $\lim_{N\ra\infty}$. Inside the
integrals in \eqref{eq:now limit inside integral}, that leaves \[
0+\frac{\zD}{\entropy}\lim_{N\ra\infty}\frac{1}{N}\sum_{n\leq N}\hw_{\widetilde{\omega\vert n}}^{k}\left(x,\ze,\frac{f\circ\mf{\widetilde{\omega\vert n}}}{\icdens\left(\mf{\widetilde{\omega\vert n}}x\right)}\right).\]

If $\alpha_{n}$ is a weakly converging sequence of signed measures
and $b_{n}$ a sequence of equicontinuous functions, then the Cesaro
average $\frac{1}{N}\sum_{n\leq N}\alpha_{\infty}\left(b_{n}\right)$
has the same limiting behavior as $\frac{1}{N}\sum_{n\leq N}\alpha_{n}\left(b_{n}\right)$.
In other words, any converging measure $\alpha_{n}$ may be replaced
by its limit $\alpha_{\infty}$, see Lemma \ref{lemresummation} below.
Let us verify the assumptions: Lemma \ref{lem:hw-converges} below
provides the geometric justification that $\alpha_{n}\left(\cdot\right)\assign\hw_{\widetilde{\omega\vert n}}^{k}\left(x,\ze,\cdot\right)$
does converge. The expression

\[
\fibreavg f{\omega}\left(\ex z\right)=\lim_{N\ra\infty}\frac{1}{N}\sum_{n\leq N}\frac{f\circ\mf{\widetilde{\omega\vert n}}\left(\ex z\right)}{\icdens\left(\mf{\widetilde{\omega\vert n}}x\right)}\]
 will be $\lim\frac{1}{N}\sum_{n\leq N}b_{n}$. It converges uniformly
in $\ex z\in\overline{\bV}\times S^{d-1}$ for almost all $\left(\omega,x\right)$:
Lemma \ref{lem:Birkhoff on test function} below provides the reduction
to the individual ergodic theorem (and a formula). The family of all
$\mfsym$-translates of $\underline{z}\mapsto f\left(\underline{z}\right)$
is equicontinuous. So we may replace $\hw_{\widetilde{\omega\vert n}}^{k}$
with $\hw_{\widetilde{\omega}}^{k}$ in \eqref{eq:now limit inside integral}.
We have achieved our goal to show the integrand limit exists: \[
\lim_{\Rvar\ra0}\frac{\hr_{\chunkh{\dots}{\Rstart}}}{\ln\frac{\Rstart}{\Rvar}\icdens}+\frac{1}{N}\frac{N}{\ln\Rstart/\Rvar}\sum_{n\leq N}\hw_{\widetilde{\omega\vert n}}^{k}\left(x,\ze,\frac{f\circ\mf{\widetilde{\omega\vert n}}}{\icdens\left(\mf{\widetilde{\omega\vert n}}x\right)}\right)=\frac{\zD}{\entropy}\hw_{\widetilde{\omega}}^{k}\left(x,\ze,\fibreavg f{\omega}\right).\]
\end{proof}
\begin{rem}
\label{rem:sharp}A converse shows the uniform integrability condition
\eqref{eq:uniformly integrand-1} likely cannot be improved. Without
it, the same proof yields for continuous $f\geq0$, \[
\liminf_{\Rvar\ra0}\frac{1}{\ln\frac{\Rstart}{\Rvar}}\int_{\mJ}\int_{0}^{\infty}\left\vert \hr_{\chunkh{\Rvar}{\Rstart}}^{k}\left(x,\ze,f\right)\right\vert \frac{d\ze}{\ze}d\mmc\left(x\right)\geq\frac{\zD}{\entropy}\int_{\mIN\times\mJ}\int_{0}^{\infty}\left\vert \hw_{\widetilde{\omega}}^{k}\left(x,\ze,\fibreavg f{\omega}\right)\right\vert \,\frac{d\ze}{\ze}d\mmbi\left(\omega,x\right).\]
Both equality holds and the lower limit exists as a proper, finite
limit, if and only if \eqref{eq:uniformly integrand-1} after replacing
$\hr^{\var}$ by $\left|\hr\right|$ and $\hw^{\var}$ by $\left|\hw\right|$
is uniformly integrable. This follows from a converse to Fatou's lemma
(\cite[Chapters 6.8 and 6.18]{Doob94MeasureTheoryMR1253752}).

Assume we may replace $C_{k}$ with $C_{k}^{\pm}$ in the continuity
assertion of Fact \ref{fac:continuity of curvature} below when proving
\eqref{eq:def hw plim} (i.e. $\lim_{n}\alpha_{n}$). Then the last
inequality improves to \begin{equation}
\liminf_{\Rvar\ra0}\frac{1}{\ln\frac{\Rstart}{\Rvar}}\int\limits _{\Rvar}^{\Rstart}\ze^{\zD-k}\int\limits _{\op{nor}\widetilde{\mJ_{\ze}}}f\left(\ex z\right)\, dC_{k}^{\pm}\left(\mJ_{\ze},\ex z\right)\frac{d\ze}{\ze}\geq\frac{\zD}{\entropy}\int\limits _{\mIN\times\mJ}\int\limits _{0}^{\infty}\hw_{\widetilde{\omega}}^{k,\pm}\left(x,\ze,\fibreavg f{\omega}\right)\,\frac{d\ze}{\ze}d\mmbi\left(\omega,x\right).\label{eq:limit variation measures}\end{equation}
Both equality holds and the lower limit exists as a proper, finite
limit, if and only if \eqref{eq:uniformly integrand-1} after replacing
$\hr^{\var}$ by $\hr^{\pm}$ and $\hw^{\var}$ by $\hw^{\pm}$ is
uniformly integrable. (This was proved for self-similar $\mJ$ in
\cite[Proposition 3.12]{BohlZaehle12CurvatureDirectionMeasures1111.4457}.)\end{rem}
\begin{proof}
(of Theorem \ref{thm:Main result}) Lemmas \ref{lem:Zygmund space condition for uniform integrability},
\ref{lem:sufficient conditions uniform integrability} derive uniform
integrability with $u$ \eqref{eq:uniformly integrand Thm} from the
other conditions with $u$ in Assumption \ref{ass:(Uniform-integrability)}.
Then Proposition \ref{pro:(core of proof)} above guarantees convergence
to the first limit formula (with $u$). In a second run, we replace
the constant $u$ by the integration variable $x$ in all the proofs
(except in Lemma \ref{lem:ergodic-dynamical-system}, where it plays
a different role). Again, the other conditions without $u$ imply
uniform integrability without $u$ and convergence to the second formula
(without $u$). It remains to show both formulas are equal.

Write $\hw_{\widetilde{\omega}}^{k,u}$ for $\hw_{\widetilde{\omega}}^{k}$,
and $\hw_{\widetilde{\omega}}^{k,x}$ when all instances of $u$ in
its definition \eqref{eq:DEF hw limiting word} are replaced with
$x$. This is the second limit formula's integrand in the theorem.
Define $\hw_{\widetilde{\omega|m}}^{k,u}$, $\hw_{\widetilde{\omega|m}}^{k,x}$
accordingly for finite words $\widetilde{\omega|m}$. Once as stated
and once with $u$ replaced, \eqref{eq:hw transforms-to hi chunk-of-curvdens}
implies \[
\hw_{\widetilde{\omega|m}}^{k,u}\left(x,\ze,f\right)=\hr_{\left(\dots\right]}^{k}\left(\mf{\widetilde{\omega|m}}x,\ze\zr{\widetilde{\omega|m}}u,f\circ\mf{\widetilde{\omega|m}}^{-1}\right)=\hw_{\widetilde{\omega|m}}^{k,x}\left(x,\ze/\zrp{\widetilde{\omega|m}}ux,f\right).\]
Then as $m\ra\infty$, Lemma \ref{lem:hw-converges} gives $\hw_{\widetilde{\omega}}^{k,u}\left(x,\ze,f\right)=\hw_{\widetilde{\omega}}^{k,x}\left(x,\ze/\zrp{\widetilde{\omega}}ux,f\right)$.
This and the coordinate transformation $\left(\omega,x,\ze\right)\mapsto\left(\omega,x,\ze/\zrp{\widetilde{\omega}}ux\right)$
turn both limit formulas into each other.
\end{proof}
The point $\left(\eta,u\right)$, the Brin group $\G_{\eta u}$, and
$\fibreavg f{\omega}$ were defined in \eqref{eq: DEF Brin group}.
Recall \[
\fibreavg f{\omega}\left(z,n\right)\assign\int\limits _{\mIN\times\mJ}\int\limits _{\G_{\eta u}}f\left(\hat{y},\left(\bplim{\eta}u^{\prime}\hat{y}\right)^{\orth-1}\hat{g}\left(\bplim{\omega}u^{\prime}z\right)^{\orth}n\right)\, d\HG\left(\hat{g}\right)d\mmc\left(\hat{y}\right)\]
 does not depend on $\left(\eta,u\right)$. 

To readers not interested in directed curvature, the next lemma merely
recalls Birkhoff's theorem.

Although the dynamical system map $\mf{\widetilde{\omega\vert m}}$
is contractive in $\overline{\bV}$, the limit $\fibreavg f{\omega}$
can depend on $z\in\overline{\bV}$ via the orbit of $n\in S^{d-1}$.
\begin{lem}
\label{lem:Birkhoff on test function}For almost all $\left(\omega,x\right)$,
the expression \begin{equation}
\lim_{M\ra\infty}\frac{1}{M}\sum_{m\leq M}\frac{f\circ\mf{\widetilde{\omega\vert m}}\left(z,n\right)}{\icdens\left(\mf{\widetilde{\omega\vert m}}x\right)}=\fibreavg f{\omega}\left(z,n\right),\label{eq:converge to finvar translate}\end{equation}
 converges uniformly in $\left(z,n\right)$, i.e. w.r.t the norm on
$\overline{\bV}\times S^{d-1}$. It can be interpreted as the $\ex{\mmbi}$-conditional
expectation of the function $a_{\left(\omega,x,z,n\right)}$ on $\ex{\msb}$-orbits
at the point $\left(\omega,x,\id\right)$, where \[
a_{\left(\tau,y,z,n\right)}\left(\omega,x,g\right)\assign\icdens\left(x\right)^{-1}f\left(x,g\left(\bplim{\tau}y^{\prime}z\right)^{\orth}n\right).\]
The limit simplifies to a constant in case $f$ does not depend on
the normal variable $n\in S^{d-1}$, \[
\fibreavg f{\omega}\left(z,n\right)=\int f\, d\mmc.\]
 \end{lem}
\begin{proof}
Recall the extended dynamical system $\left(\mIN\times\mJ\times O\left(d\right),\ex{\msb},\ex{\mmbi}\right)$,
$\ex{\mmbi}=\mmbi\otimes\mathcal{H}_{O\left(d\right)}$, \[
\ex{\msb}^{-m}\left(\omega,x,g\right)=\left(\ms^{m}\omega,\mf{\widetilde{\omega|m}}x,\left(\mf{\widetilde{\omega|m}}^{\prime}x\right)^{\orth}g\right).\]

We will rewrite $f\circ\mf{\widetilde{\omega\vert m}}$ asymptotically
in terms of this system. The base point $z$ can be moved to $x$
at the price of an additional $\bpsym$ (defined in \eqref{eq:def psi}):
\begin{eqnarray*}
f\circ\mf{\widetilde{\omega\vert m}}\left(z,n\right) & = & f\left(\mf{\widetilde{\omega\vert m}}z,\left(\mf{\widetilde{\omega\vert m}}^{\prime}z\right)^{\orth}n\right)\\
 & = & f\left(\mf{\widetilde{\omega\vert m}}z,\left(\mf{\widetilde{\omega\vert m}}^{\prime}x\right)^{\orth}\left(\bp{\widetilde{\omega\vert m}}x^{\prime}z\right)^{\orth}n\right).\end{eqnarray*}
The asymptotic approximation of $\frac{f}{\icdens}\circ\mf{\widetilde{\omega\vert m}}$
will be later obtained from \begin{equation}
a_{\left(\tau,y,z,n\right)}\circ\underline{\msb}^{-m}\left(\omega,x,g\right)=\frac{f\left(\mf{\widetilde{\omega|m}}x,\left(\mf{\widetilde{\omega|m}}^{\prime}x\right)^{\orth}g\left(\bplim{\tau}y^{\prime}z\right)^{\orth}n\right)}{\icdens\left(\mf{\widetilde{\omega|m}}x\right)}\label{eq: aux a transformed}\end{equation}
 by setting $g\assign\id$, $\left(\tau,y\right)\assign\left(\omega,x\right)$.
Indeed, the estimates $\left|\mf{\widetilde{\omega|m}}x-\mf{\widetilde{\omega|m}}z\right|\leq\rmax^{m}\diam\bV$
and $\left|\left(\bplim{\omega}x^{\prime}z\right)-\left(\bp{\widetilde{\omega|m}}x^{\prime}z\right)\right|\leq\rmax^{m\holder}c_{\bpsym}$
depend only on $m$. Since $f$ is uniformly continuous, this implies
uniform convergence \begin{equation}
\left|\frac{1}{M}\sum_{m\leq M}\left(\frac{f\circ\mf{\widetilde{\omega\vert m}}\left(z,n\right)}{\icdens\left(\mf{\widetilde{\omega\vert m}}x\right)}-a_{\left(\omega,x,z,n\right)}\circ\underline{\msb}^{-m}\left(\omega,x,\id\right)\right)\right|\underset{M\ra\infty}{\longrightarrow}0,\label{eq: f asymptotics}\end{equation}
 i.e., at a rate that does not depend on $\omega,x,z,n$. 

Next, we will apply the Birkhoff ergodic theorem to $a$. Let $\mathcal{D}$
be a countable, dense set of $\left(\tau,y,z,n\right)\in\mIN\times\mJ\times\overline{\bV}\times S^{d-1}$.
On each element of $\mathcal{D}$, the limit \begin{equation}
\lim_{M\ra\infty}\frac{1}{M}\sum_{m\leq M}a_{\left(\tau,y,z,n\right)}\circ\underline{\msb}^{-m}\left(\omega,x,g\right)\label{eq: Birkhoff f asymptotics}\end{equation}
 exists for almost every $\left(\omega,x,g\right)$. Lemma \ref{lem:ergodic-dynamical-system}
provides a formula (a $\ex{\mmbi}$-conditional expectation of $a_{\left(\tau,y,z,n\right)}$
on $\underline{\msb}$-invariant sets). Once we know it converges
uniformly in $\left(\tau,y,z,n\right)$ and $g$ at each fixed $\left(\omega,x\right)$,
we will be allowed to set $\left(\tau,y\right)=\left(\omega,x\right)$
and $g=\id$. Then \eqref{eq: f asymptotics} and \eqref{eq: Birkhoff f asymptotics}
will yield the assertion.

The auxiliary function \[
\left(\tau,y,z,n,g\right)\mapsto g\left(\bplim{\tau}y^{\prime}z\right)^{\orth}n\]
 is uniformly continuous, see Proposition \ref{pro:(Distortion-converges)}
\eqref{enu:pro:(Distortion-converges) 1 Both--converge} for $\bplim{\tau}y^{\prime}z$.
We insert this into the explicit formulas \eqref{eq: aux a transformed}
for $a\circ\underline{\msb}^{-m}$ and \eqref{eq:DEF f limit} for
$\fibreavg f{\omega}$. So both \eqref{eq: Birkhoff f asymptotics}
and $\fibreavg f{\omega}$ are uniformly equicontinuous w.r.t. $\left(\tau,y,z,n\right)$
and $g$ at each fixed $\left(\omega,x\right)$. Equality extends
to the closures $\left(\tau,y,z,n\right)\in\overline{\mathcal{D}}$,
$g\in O\left(d\right)$. The Arzela-Ascoli theorem improves point-wise
convergence on $\overline{\mathcal{D}}\times O\left(d\right)$ to
the uniform convergence used above. We have shown \begin{multline*}
\frac{1}{M}\sum_{m\leq M}\frac{f\circ\mf{\widetilde{\omega\vert m}}\left(z,n\right)}{\icdens\left(\mf{\widetilde{\omega\vert m}}x\right)}\underset{M}{\longrightarrow}\\
\int\limits _{\mIN\times\mJ}\int\limits _{\G_{\eta u}}a_{\left(\omega,x,z,n\right)}\left(\hat{\vartheta},\hat{y},\left(\bplim{\eta}u^{\prime}\hat{y}\right)^{\orth-1}\hat{h}\left(\bplim{\eta}u^{\prime}x\right)^{\orth}\id\right)\, d\HG\left(\hat{h}\right)d\mmbi\left(\hat{\vartheta},\hat{y}\right).\end{multline*}

Finally, $\icdens\left(\mf{\widetilde{\omega\vert m}}x\right)^{-1}-\icdens\left(\mf{\widetilde{\omega\vert m}}u\right)^{-1}$
asymptotically vanishes for any $x,u\in\bV$ because $\mf{\widetilde{\omega\vert m}}$
contracts and $\icdens$ is Hölder continuous. So the left side the
same for almost all values of $x$. Since the right side is continuous,
it does not depend on the choice of $x$ at all. Set $x\assign u$
and simplify $\bplim{\eta}u^{\prime}u=\id$, $\mmbi=\mmi$ lacking
$\hat{\vartheta}$, $\mmi/\icdens=\mmc$. 
\end{proof}
Recall \eqref{eq:DEF H is entropy}: $\entropy\assign-\zD\int_{\mJ}\ln\zr{x_{1}}{\ms x}\, d\mmi\left(x\right)$.
\begin{lem}
\label{lem:Birkhoff contraction ratio} Writing \[
N\left(\omega,u,\frac{\Rvar}{\ze}\right)\assign\max\left\{ n\in\mathbb{N}\,:\,\zr{\widetilde{\omega\vert n}}u>\frac{\Rvar}{\ze}\right\} ,\]
we have \[
\lim_{\Rvar\ra0}\frac{\ln\Rstart/\Rvar}{N\left(\omega,u,\Rvar/\ze\right)}=\frac{\entropy}{\zD}\]
 for $\mmbi$-almost all $\omega$ and all $u\in\mJ$.\end{lem}
\begin{proof}
See the first half of the proof of \cite[Lemma 4.7]{Patzschke04tangentMeasureSelfconformal}
for convergence and identification of the limit, and e.g. \cite[Thm. 1.9.7]{PrzytyckiUrbanski10ConformalMR2656475}
for the entropy interpretation. Bounded distortion extends it to all
$u\in\mJ$.
\end{proof}

\subsection{\label{sub:Time-averaged-distance,-distortion,}Time-averaged distance,
distortion, and ergodic fibres}

The iterated function system maps $\mJ$ onto some smaller building
block $\mf{\omega|n}\mJ$ of the fractal under the non-linear contraction
$\mf{\omega_{n}}\circ\dots\circ\mf{\omega_{1}}=\mf{\omega|n}$. But
curvature and parallel sets transform nicely under a similarity $\mf{\omega|n}^{\prime}x$.
\comment{But usually, $\mf i\left(\mJ_{\ze}\right)$ does not agree
locally with any (thinner) parallel set of $\mf i\mJ$. Parts of $\mf i\left(\mJ_{\ze}\right)$
are coated more thickly, parts less so.} To compensate, the (non-contracting)
distortion $\bp{\omega|n}x$ magnifies by one and contracts by the
other map. It usually does not converges as $n\ra\infty$ for an infinite
sequence $\omega\in\mIN$. This is why curvature densities in the
sense of \cite{RatajZaehle10CurvatureDensitiesSelfSimilarSets,BohlZaehle12CurvatureDirectionMeasures1111.4457}
do not exist. Following \cite{Zaehle01AverageDensitySelfconformalMR1825985},
the solution is to reverse the sub-words, $\widetilde{\omega|n}\assign\omega_{n}\omega_{n-1}\dots\omega_{1}$.
The two-sided right shift dynamical system inverts the sub-words:
$\msb^{-n}\left(\omega,x\right)=\left(\ms^{n}\omega,\mf{\widetilde{\omega|n}}x\right)$.
Hölder techniques will show $\bp{\widetilde{\omega|n}}x$ does converge.
Our $\lim_{n}\left|\bp{\widetilde{\omega|n}}xy\right|$ is equal to
Zähle's time-averaged distance function $\op{dist}_{\omega}\left(x,y\right)$. 

It is well-known the two-sided shift dynamical system $\mIN\times\mJ$
has an ergodic Gibbs measure $\mmbi$. But the directional nature
of our curvature requires a further extension by the orthogonal group.
We provide an explicit formula for its ergodic fibre decomposition.
A key insight, the distortion derivative $\bpsym^{\prime}$ arises
as a skewing factor, by which the fibres differ from a product measure.
(Readers interested in curvature on $\Rd$ without the normals can
ignore $\bpsym^{\prime}$ and Lemma \ref{lem:ergodic-dynamical-system}.) 

Recall \eqref{eq:def psi}, \begin{eqnarray*}
\bp{\widetilde{\omega|n}}xy & \assign & \left(\mf{\widetilde{\omega|n}}^{\prime}x\right)^{-1}\left(\mf{\widetilde{\omega|n}}y-\mf{\widetilde{\omega|n}}x\right),\end{eqnarray*}

Notice \begin{equation}
\mK^{-1}\leq\zrp{\tau}xy=\frac{\zr{\tau}y}{\zr{\tau}x}\leq\mK\label{eq: abs psi prime}\end{equation}
 and $\bp{\omega|n}xy=\left(\bp{\omega|n}u^{\prime}x\right)^{-1}\left(\bp{\omega|n}uy-\bp{\omega|n}ux\right)$.
\begin{fact}
\label{fac:rate of convergence}A sequence $a_{n}$ is said to convergence
at the rate $c_{\bpsym}s^{n}$ if $\sup_{k}\left|a_{n}-a_{n+k}\right|\leq c_{\bpsym}s^{n}$
for all $n$. Note that if $a_{n}$ and $b_{n}$ converge at the rates
$c_{a}s^{n}$ and $c_{b}s^{n}$, then the sequences $a_{n}b_{n}$
and $1/a_{n}$ converge at the rates $c_{ab}s^{n}$ and $c_{1/a}s^{n}$
respectively, with $c_{ab}\leq\left|a_{1}\right|c_{b}+\left|b_{1}\right|c_{a}+2c_{a}c_{b}s$
and $c_{1/a}\leq c_{a}/\inf_{n}\left|a_{n}\right|^{2}$. \end{fact}
\begin{prop}
(Distortion converges)\label{pro:(Distortion-converges)} For all
$x,y\in\bar{\bV}$ and all codes $\omega\in\mIN$,
\begin{enumerate}
\item \label{enu:pro:(Distortion-converges) 1 Both--converge}Both \begin{eqnarray*}
\bplim{\omega}xy & \assign & \lim_{n\ra\infty}\bp{\widetilde{\omega|n}}xy,\\
\bplim{\omega}x^{\prime}y & = & \lim_{n\ra\infty}\bp{\widetilde{\omega|n}}x^{\prime}y\end{eqnarray*}
 converge at the universal rate $c_{\bpsym}\rmax^{n\holder}$. The
constant $0<c_{\bpsym}<\infty$ does not depend on $x,y,\omega$.
Thus both limiting functions are uniformly continuous in $\omega,x,y$.
(The derivative of the limit $\bplim{\omega}x$ agrees with the limit
of the derivative $\bp{\widetilde{\omega|n}}x^{\prime}$.)
\item \label{enu: pro:(Distortion-converges) 2 BDP}There is a constant
$\mK_{\bpsym}$ such that \[
\mK_{\bpsym}^{-1}\leq\frac{\left|\bp{\widetilde{\omega|n}}xy\right|}{\left|x-y\right|}\leq\mK_{\bpsym}.\]

\item \label{enu:pro:(Distortion-converges) 3 similarity}$\bp{\omega|n}x\circ\mf{\omega|n}^{-1}$
is a similarity with Lipschitz constant $\zr{\omega|n}x^{-1}$.
\item \label{enu:pro:(Distortion-converges) 4 gauge}$\bp{\widetilde{\omega|n}}xx=0$
and $\bp{\widetilde{\omega|n}}x^{\prime}x=\op{id}$.
\end{enumerate}
\end{prop}
\begin{proof}
For convenience, set $\tau\assign\widetilde{\omega|n}$. 

For the second assertion, note \cite[Lemma 2]{Patzschke97ConformalMultifractalMR1479016}
states \[
\op{const}\leq\frac{\left|\mf{\widetilde{\omega|n}}y-\mf{\widetilde{\omega|n}}x\right|}{\left\Vert \mf{\widetilde{\omega|n}}^{\prime}\right\Vert \left|x-y\right|}\leq\op{const}.\]
(Here, we made use of the $C^{1+\holder}$ conformal extension of
$\mf i$ to $\overline{\mV}$, subject to \eqref{eq:BDP.3}.) Then
$\left|\left(\mf{\widetilde{\omega|n}}^{\prime}x\right)^{-1}\left(\mf{\widetilde{\omega|n}}y-\mf{\widetilde{\omega|n}}x\right)\right|=\left|\mf{\widetilde{\omega|n}}y-\mf{\widetilde{\omega|n}}x\right|/\zr{\widetilde{\omega|n}}x$
and bounded distortion \eqref{eq:BDP.3} complete its proof.

Next, we will show the first assertion. Convergence of $\left|\bp{\tau}xy\right|$
is shown in \cite[Theorem 1.(i)]{Zaehle01AverageDensitySelfconformalMR1825985}.
Though stated only for $x,y\in\mX$, the proof works on all of $\overline{\bV}$.
It could be modified to account for pairwise distances, but for $\zrp{\tau}xy$
we need an approach inspired by cross ratios. The equality \[
\left|\bp{\tau}uv\right|=\zr{\tau}u^{-1}\left|\mf{\tau}u-\mf{\tau}v\right|\]
 permits us to rewrite\begin{equation}
\zrp{\tau}xy=\frac{\zr{\tau}y}{\zr{\tau}x}=\frac{\left|\mf{\tau}y-\mf{\tau}z\right|/\zr{\tau}z}{\left|\mf{\tau}z-\mf{\tau}y\right|/\zr{\tau}y}\,\frac{\zr{\tau}z}{\zr{\tau}x}=\frac{\left|\bp{\tau}zy\right|}{\left|\bp{\tau}yz\right|}\frac{\left|\bp{\tau}xz\right|}{\left|\bp{\tau}zx\right|}\label{eq:zrp in terms of distances}\end{equation}
 in terms of quantities which are known to converge individually at
the desired rate. It remains to find uniform bounds on all the factors.
The as-yet-undetermined variable $z$ was introduced to bound the
denominators away from zero: since $\overline{\bV}$ is connected,
we can always find some $z\in\overline{\bV}$ such that both $\left|x-z\right|$
and $\left|y-z\right|$ exceed $\op{diam}\overline{\bV}/4$. Our second
assertion then places bounds on the $\left|\bp{\tau}vu\right|$-like
expressions. At the price of enlarging the constant $c_{\bpsym}$,
this gives us both the asserted (rate of) convergence of, and bounds
on $\zrp{\tau}xy$. 

The next goal is convergence of $\left(\bp{\tau}x^{\prime}y\right)^{\orth}$.
Each of the finitely many functions $x\mapsto\mf i^{\prime}x$ is
$\holder$-Hölder continuous. Since $\inf_{i,x}\zr ix>0$, $\left(\mf i^{\prime}x\right)^{\orth}=\zr ix^{-1}\mf i^{\prime}x$
also is Hölder with respect to the action invariant metric on $O\left(d\right)$.
Application of the triangle inequality, action invariance, and Hölder
continuity let us demonstrate the Cauchy sequence property (suppressing
{}``$\orth$''): \begin{eqnarray*}
 &  & \phantom{\sum_{i=n}^{n+k-1}}d_{O\left(d\right)}\left(\left(\mf{\widetilde{\omega|n+k}}^{\prime}x\right)^{-1}\left(\mf{\widetilde{\omega|n+k}}^{\prime}y\right),\left(\mf{\widetilde{\omega|n}}^{\prime}x\right)^{-1}\left(\mf{\widetilde{\omega|n}}^{\prime}y\right)\right)\\
 & \leq & \sum_{i=n}^{n+k-1}d_{O\left(d\right)}\left(\left(\mf{\widetilde{\omega|i+1}}^{\prime}x\right)^{-1}\left(\mf{\widetilde{\omega|i+1}}^{\prime}y\right),\left(\mf{\widetilde{\omega|i}}^{\prime}x\right)^{-1}\left(\mf{\widetilde{\omega|i}}^{\prime}y\right)\right)\\
 & \leq & \sum_{i=n}^{n+k-1}d_{O\left(d\right)}\left(\left(\mf{\omega_{i+1}}^{\prime}\mf{\widetilde{\omega|i}}x\right)^{-1}\left(\mf{\omega_{i+1}}^{\prime}\mf{\widetilde{\omega|i}}y\right),\id\right)\\
 & \leq & \sum_{i=n}^{n+k-1}\op{const}\rmax^{i\holder}\leq\op{const}^{\prime}\rmax^{n\holder}.\end{eqnarray*}
By multiplying this with $\zrp{\tau}xy$, we have proved $\bp{\tau}x^{\prime}y$
itself converges at the desired rate. Since $\bp{\tau}xx=0$ and $\overline{\bV}$
is bounded and connected, we can represent $\bp{\tau}xy$ as an integral
of $\bp{\tau}x^{\prime}$ along a curve from $x$ to $y$ to see it
converges as asserted.\comment{\end{proof}
\begin{rem}
For completeness: the proposition (minus the rate of convergence)
also holds for the reverse ordering \begin{eqnarray*}
\ap{\omega|n}xy\assign\left(\mf{\omega|n}y-\mf{\omega|n}x\right)\left(\mf{\omega|n}^{\prime}x\right)^{-1}, &  & \aplim{\omega}xy\assign\lim_{n\ra\infty}\ap{\widetilde{\omega|n}}xy.\end{eqnarray*}
The identities \[
\left|\ap{\widetilde{\omega|n}}uy-\ap{\widetilde{\omega|n}}ux\right|=\left|\bp{\widetilde{\omega|n}}u^{\prime}x\right|^{-1}\left|\bp{\widetilde{\omega|n}}uy\right|,\]
 \[
\left\langle a\middle|b\right\rangle =\frac{1}{2}\left(\left|a-0\right|^{2}+\left|b-0\right|^{2}-\left|a-b\right|^{2}\right)\]
 establish convergence of $\apsym$ modulo the orthogonal group. The
family $\ap{\widetilde{\omega|n}}x$ is $n$-equi\-continuous due
to bounded distortion, which carries over to $\ap{\widetilde{\omega|n}}x^{\prime}$
because it is conformal. Arzela-Ascoli and the closedness of the differential
operator imply \[
\left(\lim\ap{\widetilde{\omega|n}}x\right)^{\prime}x=\lim\ap{\widetilde{\omega|n}}x^{\prime}x=\id.\]
 This fixes the orthogonal group gauge.\end{rem}
\begin{proof}
}
\end{proof}
Recall the shift dynamical system on $\mIN\times\mJ\times O\left(d\right)$
is given by $\ex{\mmbi}=\mmbi\otimes\text{Haar}$ and \[
\ex{\msb}^{n}\left(\omega,x,g\right)=\left(\widetilde{x|n}.\omega,\,\ms^{n}x,\,\left(\mf{x|n}^{\prime}\ms^{n}x\right)^{\orth-1}g\right).\]
Let $\left(\eta,u\right)\in\mIN\times\mJ$ be the reference point,
and recall \[
\G_{\eta u}=\overline{\left<\left(\bplim{\eta}u^{\prime}x\right)^{\orth-1}\left(\mf i^{\prime}x\right)^{\orth}\left(\bplim{\eta}u^{\prime}\mf ix\right)^{\orth}\,:\, x\in\mJ,i\in\mI\right>_{O\left(d\right)}}.\]

\begin{lem}
\label{lem:ergodic-dynamical-system}(Ergodic fibers of group extended
shift) Let $f:\mIN\times\mJ\times O\left(d\right)\ra\mathbb{R}$ be
measurable and $\ex{\mmbi}$-integrable. The Birkhoff limit is almost
everywhere given by\begin{multline}
\lim_{N\ra\infty}\frac{1}{N}\sum_{n<N}f\circ\ex{\msb}^{n}\left(\omega,x,g\right)=\\
\int\limits _{\mIN\times\mJ}\int\limits _{\G_{\eta u}}f\left(\hat{\upsilon},\hat{y},\left(\bplim{\eta}u^{\prime}\hat{y}\right)^{\orth-1}\hat{h}\left(\bplim{\eta}u^{\prime}x\right)^{\orth}g\right)\, d\HG\left(\hat{h}\right)d\mmbi\left(\hat{\upsilon},\hat{y}\right),\label{eq: projection onto ergodic fibres}\end{multline}
 where $\HG$ is the Haar probability on $\G_{\eta u}$. (It has an
interpretation as the $\ex{\mmbi}$-conditional expectation of $f$
on $\underline{\msb}$-invariant sets.)\end{lem}
\begin{proof}
Firstly, we will summarize the theory of Brin groups according to
\cite[Section 2.2]{Dolgopyat02MixingOfCompactExtensionsHyperbolicMR1919377}
because that work does not assume a smooth manifold. The shift space
$\left(\mI^{\mathbb{Z}},\msb,\mmbi\right)$ can be extended by a Hölder
function $\tau:\mI^{\mathbb{Z}}\ra H$ with values in a compact, connected
Lie group $H$. The skew product $T\left(w,g\right)=\left(\msb w,\tau\left(w\right)g\right)$
defines the extended dynamical system is $\left(\mI^{\mathbb{Z}}\times\G,T,\mmbi\otimes\text{Haar}_{H}\right)$.
Write $\tau_{n}\left(\omega,x\right)\assign\tau\left(\msb^{n-1}w\right)\dots\tau\left(w\right)$
and $\tau_{n}^{-1}\left(\omega,x\right)$ for its point-wise inverse
in $H$. The stable sets, unstable sets, and orbits are defined by
\begin{eqnarray*}
W^{s}\left(w,g\right) & = & \left\{ \left(\overline{w},\left[\lim_{n\ra\infty}\tau_{n}^{-1}\left(\overline{w}\right)\tau_{n}\left(w\right)\right]g\right)\,:\, w_{i}=\overline{w}_{i},i\geq N\,\text{for some }N\right\} ,\\
W^{u}\left(w,g\right) & = & \left\{ \left(\overline{w},\left[\lim_{n\ra\infty}\tau_{n}\msb^{-n}\left(\overline{w}\right)\tau_{n}^{-1}\msb^{-n}\left(w\right)\right]g\right)\,:\, w_{i}=\overline{w}_{i},i\leq-N\,\text{for some }N\right\} ,\\
W^{o}\left(w,g\right) & = & \left\{ \left(\msb^{n}\left(w\right),\left[\tau_{n}\left(w\right)\right]g\right),\,\left(\msb^{-n}\left(w\right),\left[\tau_{n}^{-1}\msb^{-n}\left(w\right)\right]g\right)\,:\, n\in\mathbb{N}\right\} .\end{eqnarray*}
 A t-chain is a finite sequence $W=\left(w_{\left(1\right)}g_{\left(1\right)},\dots,w_{\left(n\right)}g_{\left(n\right)}\right)$
from $\mI^{\mathbb{Z}}\times H$ with $\left(w_{\left(i+1\right)},g_{\left(i+1\right)}\right)\in W^{s}\left(w_{\left(i\right)},g_{\left(i\right)}\right)\cup W^{u}\left(w_{\left(i\right)},g_{\left(i\right)}\right)$,
and an e-chain also allows $W^{o}\left(w_{\left(i\right)},g_{\left(i\right)}\right)$.
The holonomy element of $W$ is $g\left(W\right)=g_{\left(n\right)}g_{\left(1\right)}^{-1}$;
it depends only on the $\mI^{\mathbb{Z}}$ part in case of a t-chain.
The Brin ergodicity group $\Gamma_{e}\left(w\right)\subseteq H$ at
$w$ is generated set-theoretically by all e-chains that start and
end in $w$.

The system is called reduced if any two points in $\mI^{\mathbb{Z}}$
can be connected by a t-chain $W$ with trivial holonomy element $g\left(W\right)=\id$.
Then $\Gamma_{e}$ is generated set-theoretically by all e-chains
without any restrictions. Its closure $\overline{\Gamma_{e}}\subseteq H$
is generated by all $\tau\left(w\right)$, $w\in\mI^{\mathbb{Z}}$.
The extended dynamical system is ergodic if and only if $H=\overline{\Gamma_{e}}$. 

Our goal is to reduce it. An active coordinate transformation $\Phi$
of $\mI^{\mathbb{Z}}\times H$, \[
\Phi\left(w,g\right)=\left(w,\alpha\left(w\right)g\right)=\left(w^{\prime},g^{\prime}\right)\]
 conjugates $T$ to $T^{\prime}$ (i.e., $\Phi\circ T=T^{\prime}\circ\Phi$)
if and only if it maps $\tau_{n}$ to \[
\tau_{n}^{\prime}\left(w^{\prime}\right)=\alpha\left(\msb^{n}w\right)\tau_{n}\left(w\right)\alpha^{-1}\left(w\right).\]
Thus $g^{\prime}\left(W\right)=\alpha\left(w_{\op{end}}\right)g\left(W\right)\alpha^{-1}\left(w_{\left(1\right)}\right)$.
Denote $\left[u,v\right]$ the local product, i.e., $u$ at negative
and $v$ at positive indices. We pick a reference point $w^{0}$,
connect it to $w$ by the t-chain $W_{w}=\left(w,\left[w^{0},w\right],w^{0}\right)$,
and set $\alpha\left(w\right)\assign g^{-1}\left(W_{w}\right)$, i.e.,
\[
\alpha\left(w\right)=\left[\lim_{n\ra\infty}\tau_{n}^{-1}\left(w\right)\tau_{n}\left(\left[w^{0},w\right]\right)\right]\left[\lim_{n\ra\infty}\tau_{n}\left(\msb^{-n}\left[w^{0},w\right]\right)\tau_{n}^{-1}\left(\msb^{-n}w^{0}\right)\right].\]
Our coordinate transformation $\Phi$ reduces the system because $g^{\prime}\left(W_{w}\right)=\id$.
Therefore, the dynamical subsystem $\left(\mI^{\mathbb{Z}}\times\overline{\Gamma_{e}},T^{\prime},\mmbi\otimes\mathbf{\H_{\overline{\Gamma_{e}}}}\right)$
is ergodic, where $\H_{\overline{\Gamma_{e}}}$ is the Haar probability
on $\overline{\Gamma_{e}}$. 

Secondly, we will more explicitly compute the ergodic fibres of the
original dynamical system. (Our $\alpha$ corresponds to $h$ of \cite[Theorem 3]{ParryPollicott97LivsicCompactCocycleMR1489146}
and $u$ of \cite[Theorem 1.3]{Raugi07ErgodicSkewProductsMR2430192}
and \cite[(10)]{ConzeRaugi09ErgodicDecompositionCocycleMR2539552},
in the sense that $\overline{\Gamma_{e}}\alpha$ agrees with its counterparts.)
In the original coordinates, $\overline{\Gamma_{e}}$ and $\alpha$
depend on the choice of reference point: \[
\overline{\Gamma_{e}}=\overline{\left<\alpha\left(\msb\hat{w}\right)\,\tau\left(\hat{w}\right)\,\alpha^{-1}\left(\hat{w}\right)\,:\,\hat{w}\in\mI^{\mathbb{Z}}\right>_{H}}.\]
For all $h\in H$, define the function \[
f_{h}\left(w,g\right)\assign f\left(w,gh\right)\]
on $\mI^{\mathbb{Z}}\times H$. Birkhoff's ergodic theorem for the
ergodic, reduced subsystem $\mI^{\mathbb{Z}}\times\overline{\Gamma_{e}}$
implies that $f_{h}\circ\Phi^{-1}$ converges on $\H_{\overline{\Gamma_{e}}}$-almost
all $\alpha\left(w\right)g\in\overline{\Gamma_{e}}$: \[
\frac{1}{n}\sum_{l<n}f\circ T^{l}\left(w,gh\right)=\frac{1}{n}\sum_{l<n}\left[f_{h}\circ\Phi^{-1}\right]\circ\left(T^{\prime}\right)^{l}\circ\Phi\left(w,g\right)\underset{n\ra\infty}{\longrightarrow}\]
 \[
\int_{\mI^{\mathbb{Z}}}\int_{\overline{\Gamma_{e}}}f_{h}\circ\Phi^{-1}\left(\hat{w},\hat{g}\right)d\H_{\overline{\Gamma_{e}}}\left(\hat{g}\right)d\mmbi\left(\hat{w}\right)=\int_{\mI^{\mathbb{Z}}}\int_{\overline{\Gamma_{e}}}f\left(\hat{w},\alpha\left(\hat{w}\right)^{-1}\hat{g}\alpha\left(w\right)gh\right)d\H_{\overline{\Gamma_{e}}}\left(\hat{g}\right)d\mmbi\left(\hat{w}\right).\]
We have used $T=\Phi^{-1}\circ\Phi\circ T=\Phi^{-1}\circ T^{\prime}\circ\Phi$
to insert $\Phi$ into the first line. In the last equality, we have
replaced $h$ with $\alpha\left(w\right)gh$ because $\H_{\overline{\Gamma_{e}}}$
is action invariant.

Birkhoff's theorem can also be applied to $f$ on the original dynamical
system $\mI^{\mathbb{Z}}\times H$. The above ergodic sum converges
to the ergodic fibre average for $\H_{H}$-almost every $gh$. We
still need to show $\H_{H}$-almost every $gh\in H$ can be written
using an $h$-dependent, $\H_{\overline{\Gamma_{e}}}$-full set of
$\alpha\left(w\right)g\in\overline{\Gamma_{e}}$. But Federer's coarea
formula \cite[Theorem 5.4.9]{KrantzParks08GeometricIntegrationTheoryMR2427002}
proves $\H_{H}$ can be obtained by integrating right translates of
$\alpha\left(w\right)^{-1}\H_{\overline{\Gamma_{e}}}$ because any
Haar probability is a Hausdorff measure.

Thirdly, we will apply the above theory. The positive (right) index
part of $\left(\omega,x\right)\in\mI^{\mathbb{Z}}$ will be given
by $x\in\mJ$ and the negative (left) one by $\omega\in\mIN$. Both
the thermodynamic potential $\ln\zr{x|1}{\ms x}$ and $\left(\mf{x|1}^{\prime}\ms x\right)^{\orth}$
are Lipschitz continuous on $H\assign O\left(d\right)$ with respect
to \[
d_{\theta}\left(\left(\omega,x\right),\left(\eta,u\right)\right)=\max\left\{ \theta^{n}\,:\,\omega_{n}\neq\eta_{n}\,\text{or }x_{n}\neq u_{n}\right\} \]
 if $\theta\assign\rmax^{\holder}$, as required. We set \[
\tau_{n}\left(\omega,x\right)\assign\left(\mf{x|n}^{\prime}\ms^{n}x\right)^{\orth-1}=\left(\mf{x_{n}}^{\prime}\ms^{n}x\right)^{\orth-1}\dots\left(\mf{x_{1}}^{\prime}\ms x\right)^{\orth-1},\]
 and $\tau_{n}\msb^{-n}\left(\omega,x\right)=\left(\mf{\widetilde{\omega|n}}^{\prime}x\right)^{\orth-1}$.
Let $\left(\eta,u\right)$ be our reference point $w^{0}$. The t-chain
$\left(w,\left[w^{0},w\right],w^{0}\right)$ becomes $\left(\left(\omega,x\right),\left(\eta,x\right),\left(\eta,u\right)\right)$.
Note $\tau_{n}$ depends only on $x$, so the stable set is flat,
i.e., $\left(\eta,x,g\right)\in W^{s}\left(\omega,x,g\right)$. We
compute \[
\alpha\left(\omega,x\right)=\left[\lim_{n\ra\infty}\id\right]\,\left[\lim_{n\ra\infty}\left(\mf{\widetilde{\eta|n}}^{\prime}x\right)^{\orth-1}\left(\mf{\widetilde{\eta|n}}^{\prime}u\right)^{\orth}\right]=\left(\bplim{\eta}u^{\prime}x\right)^{\orth}\]
 (by Proposition \ref{pro:(Distortion-converges)}) and use the formulas. 
\end{proof}

\subsection{Renewal radii and covariant neighborhood net\label{sub:neighborhood-net}}

The next two sections have more geometrical content. They generalize
the methods from the self-similar case (\cite{BohlZaehle12CurvatureDirectionMeasures1111.4457,RatajZaehle10CurvatureDensitiesSelfSimilarSets})
and unify several notions of curvature densities (Remark \ref{rem:Our-A-unifies}).

We will later split the $d\ze$-integral in \eqref{eq:integral-of-curvature-density}
into chunks of comparable level of detail of $\mJ$ {}``visible''
through the parallel set $\mJ_{\ze}$. More precisely, near a point
$x\in\mJ$, the parallel set width $\R\left(x,m\right)$ is {}``well-adapted''
to magnification under $\mf{x|m}^{-1}$. Intuitively, $\R\left(x,m\right)\sim\dist\left(x,\mf{x|m}\mX^{c}\right)$.
{}``Well-adapted'' means we can localize the curvature measures
because no other level set $\left(\mf{\tau}\mJ\right)_{\R\left(x,m\right)}$,
$\tau\in\mI^{m}$, overlaps $\left(\mf{x|m}\mJ\right)_{\R\left(x,m\right)}$
in an $\R\left(x,m\right)$-ball around $x$. This requirement and
covariance under $\mf{x|m}$ naturally lead to our definition of the
renewal radius $\R\left(x,m\right)$. For technical reasons, we will
integrate a tent function $\Along$ around $x$ instead the curvature
of the $\R\left(x,m\right)$-ball itself. The support of $z\mapsto\A xz{\R\left(x,m\right)}$
is contained in a slightly larger ball because $\Along$ needs to
sample an environment of $\partial\left(\mJ_{\ze}\right)$.

Recall $1<a$ and $0<\Rstart\leq a^{-1}\dist\left(\mJ,\bV^{c}\right)$.
\begin{defn}
\label{def:R, N} Define for $\mmc$-almost all $x\in\mJ$ and $m\in\mathbb{N}$:
$\R\left(x,-1\right)\assign\Rstart$, \begin{eqnarray*}
\widetilde{\R}\left(x,0\right) & \assign & \frac{\dist\left(x,X^{c}\right)}{2\Kf a}\min\left\{ 1,\frac{2\Kf^{-1}a\Rstart}{\max_{x\in\mJ}\dist\left(x,\mX^{c}\right)}\right\} ,\\
\R\left(x,m\right) & \assign & \sup_{n\geq m}\zr{x|n}{\ms^{n}x}\widetilde{\R}\left(\ms^{n}x,0\right).\end{eqnarray*}
We could eliminate the minimum in the definition of $\widetilde{\R}\left(x,0\right)$
if we replaced the (OSC) set $\mX$ with its iterated image under
the IFS.\end{defn}
\begin{lem}
\label{lem:(Renewal-radii)}(Renewal radii) For $\mmc$-almost all
$x,y\in\mJ$, all $m\geq0$, $z\in\bV$: 
\begin{enumerate}
\item \label{enu:lem:(Renewal-radii) 4 intersection mJ}Avoids overlap sets
$\left(\mf{\tau}\mJ\right)_{\ze}\cap\left(\mf{\vartheta}\mJ\right)_{\ze}$
($\tau,\vartheta\in\mI^{m}$): \begin{eqnarray*}
\left|z-x\right|<a\R\left(x,m\right) & \Longrightarrow & z\notin\left(\mf{x|m}\mX^{c}\right)_{\R\left(x,m\right)},\\
\left|y-x\right|<2a\R\left(x,m\right) & \Longrightarrow & x|m=y|m,\end{eqnarray*}

\item \label{enu:lem:(Renewal-radii) covariance}Covariance: \textup{$\R\left(x,m\right)=\zr{x|m}{\ms^{m}x}\R\left(\ms^{m}x,0\right)$,}
\item \label{enu:lem:(Renewal-radii) 1 monotonicity}Monotonicity: $a^{-1}\op{dist}\left(\mJ,\bV^{c}\right)\geq\Rstart\geq\R\left(x,m\right)\geq\R\left(x,m+1\right)>0$,
\item \label{enu:lem:(Renewal-radii)  2 limit}Limit: $\R\left(x,m\right)\underset{m\ra\infty}{\longrightarrow}0$,
\item \textup{\label{enu:lem:(Renewal-radii) 6 integrable}}Integrability:
$\int_{\mJ}\left|\ln\R\left(y,0\right)\right|\, d\mmc\left(y\right)<\infty$
and, for some $c>1$: \begin{equation}
\int_{\mJ}c^{\left|\ln\R\left(y,0\right)\right|}d\mmc\left(y\right)<\infty.\label{eq:logdist is very integrable}\end{equation}
 
\end{enumerate}
\end{lem}
\begin{proof}
(Compare \cite{Patzschke04tangentMeasureSelfconformal}.) The first
assertion holds if any $y\in\bV$ avoids $\mf{x|m}\mX^{c}$ if $\left|y-x\right|<2a\R\left(x,m\right)$.
Indeed: given $z$, set $y\assign2z-x$. Given $y\in\mJ$, (OSC) implies
$y\in\mf{x|m}\mJ$.

We use the Open Set Condition, the triangle inequality, bounded distortion
\eqref{eq:BDP.3} and an undetermined $0<c_{\R}\leq1$:\begin{eqnarray*}
\dist\left(y,\mf{x|m}\mX^{c}\right) & \overset{\text{(OSC)}}{=} & \sup_{n\geq m}\dist\left(y,\mf{x|n}\mX^{c}\right)\geq\sup_{n\geq m}\dist\left(x,\mf{x|n}\mX^{c}\right)-\left|x-y\right|\\
 & > & \sup_{n\geq m}\dist\left(\mf{x|n}\ms^{n}x,\mf{x|n}\mX^{c}\right)-2a\R\left(x,m\right)\\
 & \overset{\eqref{eq:BDP.3}}{\geq} & \Kf^{-1}\sup_{n\geq m}\zr{x|n}{\ms^{n}x}\dist\left(\ms^{n}x,\mX^{c}\right)-\frac{2a}{c_{\R}}\R\left(x,m\right)\assign0.\end{eqnarray*}
 We solve the last equation for $\R\left(x,m\right)$, \[
\R\left(x,m\right)=\frac{c_{\R}}{2a\Kf}\sup_{n\geq m}\zr{x|n}{\ms^{n}x}\dist\left(\ms^{n}x,\mX^{c}\right),\]
 and recuperate the original definition of $\R\left(x,m\right)$ if
$c_{\R}\leq1$ is the largest value such that $\max_{x\in\mJ}\R\left(x,0\right)\leq\Rstart$.
(As an aside, $\widetilde{\R}\left(x,0\right)\leq\R\left(x,0\right)\leq\Kf^{2}\widetilde{\R}\left(x,0\right)$.)

The second assertion (covariance) follows from, ($n\geq m$) \[
\zr{x|n}{\ms^{n}x}=\zr{x|m}{\ms^{m}x}\zr{\left(\ms^{m}x\right)|n-m}{\ms^{n-m}\left(\ms^{m}x\right)},\]
 and the definitions of $\R\left(x,m\right)$ and $\R\left(\ms^{m}x,0\right)$.
Monotonicity is trivial except for $\R\left(x,m\right)>0$, which
follows from the integrability. The limit assertion stems from \[
\R\left(x,m\right)=\zr{x|m}{\ms^{m}x}\R\left(\ms^{m}x,0\right)\leq\rmax^{m}\Rstart.\]
Integrability of $c^{\left|\ln\R\left(y,0\right)\right|}$ is reduced
by $c^{\prime}\dist\left(y,X^{c}\right)=\widetilde{\R}\left(y,0\right)\leq\R\left(y,0\right)\leq\Rstart$
to \begin{equation}
\int c^{\left|\ln\,\dist\left(y,X^{c}\right)\right|}d\mmc\left(y\right)<\infty.\end{equation}
 \cite[Lemma 3.8]{Patzschke04tangentMeasureSelfconformal} proves
$\int\left|\ln\,\dist\left(\cdot,X^{c}\right)\right|d\mmc$ is finite
using (SOSC). Our case is similar. Patzschke dominates the integral
by a geometric series. We can safely increase its base while (in Patzschke's
notation) $0<\ln c<\left|\ln\left(C_{2}^{-1}\left\Vert S_{\eta}^{\prime}\right\Vert r_{\min}\right)\right|\left|\ln\left(1-C_{5}^{-1}\psi\left[\eta\right]\right)\right|$.
\end{proof}
The purpose of $\Along$ is to sample the curvature located in a small
ball $B\left(x,a\ze\right)$, $\ze<\R\left(x,m\right)$. It has to
work together with weak convergence of curvature measures (Section
\ref{sub:Curvature-of-distorted}). So we define it as a continuous
counterpart of $1_{B\left(x,a\ze\right)}\left(z\right)$: the tent
function \[
\A xz{\ze}\assign\max\left(1-\frac{\left|x-z\right|}{\ze a},\,0\right).\]
Many other choices are viable due to \eqref{eq:A cancels}. So we
will treat $\Along$ as a black box and demand only the next lemma
holds.
\begin{rem}
\label{rem:Our-A-unifies}Our $\Along$ unifies previous approaches
in the self-similar setting. It can be replaced with any of the locally
homogeneous neighborhood nets $A_{F}^{\text{\cite{RatajZaehle10CurvatureDensitiesSelfSimilarSets}}}\left(x,\ze\right)$
from \cite[Example 2.1.1]{RatajZaehle10CurvatureDensitiesSelfSimilarSets}
if $\A xz{\ze}\assign1_{A_{F}^{\text{\cite{RatajZaehle10CurvatureDensitiesSelfSimilarSets}}}\left(x,\ze\right)}\left(z\right)$.
(Our continuity is replaced with their Lemma 2.1.3.) This includes
a net that eliminates the denominator $\Along$ in $\hr$ and $\hw$. 
\end{rem}
Items \ref{enu:lem:(neighborhood-net) Measurability}, \ref{enu:lem:(neighborhood-net) D-set-property}
are designed to enable Fubini (Proposition \ref{pro:Fubini}); \ref{enu:lem:(neighborhood-net) Locality},
\ref{enu:lem:(neighborhood-net) ball chain}, \ref{enu:lem:(neighborhood-net) Covariance-w.r.t-the distortion}
to localize and preimage curvature (Lemma \ref{lem:preimages of hr, hw curvdens});
and \ref{enu:lem:(neighborhood-net) D-set-property}, \ref{enu:lem:(neighborhood-net) Lipschitz-continuity}
to preserve weak convergence (Lemma \ref{lem:hw-converges}).
\begin{lem}
\label{lem:(neighborhood-net)}(Covariant neighborhood net) For almost
all $\xhat,\xplain,\yhat,\yplain\in\mJ$, all $u,\zplain,\zhat\in\bV$,
and almost all $\ze,\rplain\leq\Rstart$:
\begin{enumerate}
\item \label{enu:lem:(neighborhood-net) Measurability}(Measurability) $\left(x,z,\ze\right)\mapsto\A xz{\ze}$
is Borel measurable; $0\leq\A xz{\ze}\leq1$.
\item \label{enu:lem:(neighborhood-net) Locality}(Locality) \[
\left\{ \zplain\in\Rd:\A{\xplain}{\zplain}{\rplain}\neq0\right\} \subseteq\Int B\left(\xplain,a\rplain\right)\subseteq\bV\]
.
\item \label{enu:lem:(neighborhood-net) D-set-property}($\zD$-set property)
There is a constant $C_{\Along}$, $0<C_{\Along}<\infty$, such that
\[
C_{\Along}^{-1}\leq\frac{\int\A{\yplain}{\zplain}{\rplain}d\mmc\left(\yplain\right)}{\rplain^{\zD}}\leq C_{\Along}.\]
 More generally, let write $\tau\in\mI^{*}$ and $\bp{\tau}u\zhat\in\left(\left(\bp{\tau}u\mJ\right)_{\rhat}\right)$.
Then \textup{\begin{equation}
C_{\Along}^{-1}\leq\frac{\int\A{\bp{\tau}u\yhat}{\bp{\tau}u\zhat}{\rhat}d\mmc\left(\yhat\right)}{\rhat^{\zD}}\leq C_{\Along}.\label{eq:Along -- constant CA}\end{equation}
}
\item \label{enu:lem:(neighborhood-net) ball chain}(2-chain of supports
avoids intersections $\left[\mf{\tau}\mJ\right]_{\rplain}\cap\left[\mf{\xplain|m}\mJ\right]_{\rplain}$)
If $\rplain\leq\R\left(\xplain,m\right)$, \[
\A{\xplain}{\zplain}{\rplain}\neq0\text{ and }\A{\yplain}{\zplain}{\rplain}\neq0\,\Longrightarrow\xplain|m=\yplain|m.\]

\item \label{enu:lem:(neighborhood-net) Covariance-w.r.t-the distortion}(Covariance
w.r.t the $\bpsym$-distorted distance; compatibility with the dynamics)\\
 If $\rplain\leq\R\left(\xplain,m\right)$ and $\mf{\xplain|m}\zhat\in\left(\left(\mf{\xplain|m}\mJ\right)_{\rplain}\right)$,
then \[
\A{\yplain}{\mf{\xplain|m}\zhat}{\rplain}=\A{\bp{\xplain|m}u\ms^{m}\yhat}{\bp{\xplain|m}u\zhat}{\frac{\rplain}{\zr{\xplain|m}u}}1\left\{ \yplain|m=\xplain|m\right\} .\]

\item \label{enu:lem:(neighborhood-net) Lipschitz-continuity}(Lipschitz
continuity) As a function $\left|x-z\right|\mapsto\A xz{\ze}$ for
fixed $\ze$, \[
\op{Lip}\left(\Along\right)\leq\ze^{-1}.\]
 
\end{enumerate}
\end{lem}
\begin{proof}
Point one and six are obvious. The second and fourth follow from Lemma
\ref{lem:(Renewal-radii)} point \eqref{enu:lem:(Renewal-radii) 4 intersection mJ}. 

The third assertion will be proved in the version with $\bp{\tau}u$.
It contains the simpler version for $\tau=\varnothing$, $\bp{\tau}u=\op{id}$.
Abbreviate $\bps{\tau}u\assign\bp{\tau}u$. We want to estimate\global\long\def\lev{t}
 \begin{eqnarray*}
\int_{\mJ}\A{\bps{\tau}u\yhat}{\bps{\tau}u\zhat}{\rhat}\, d\mmc\left(\yhat\right) & = & \int_{0}^{1}\int_{\mJ}1\left\{ 1-\lev\geq\A{\bps{\tau}u\yhat}{\bps{\tau}u\zhat}{\rhat}\right\} \, d\mmc\left(\yhat\right)d\lev\\
 & = & \int_{0}^{1}\int_{\mJ}1\left\{ \lev a\rhat\geq\left|\bps{\tau}u\yhat-\bps{\tau}u\zhat\right|\right\} \, d\mmc\left(\yhat\right)d\lev.\end{eqnarray*}
 There is a point $\xhat_{\zhat}\in\mJ$ such that $\left|\bps{\tau}u\zhat-\bps{\tau}u\xhat_{\zhat}\right|\leq\rhat$.
The triangle inequality implies\begin{align*}
\left|\bps{\tau}u\yhat-\bps{\tau}u\zhat\right|-\rhat & \leq & \left|\bps{\tau}u\yhat-\bps{\tau}u\xhat_{\zhat}\right| & \leq & \left|\bps{\tau}u\yhat-\bps{\tau}u\zhat\right|+\rhat,\\
\left\{ \lev a\rhat+\rhat\geq\left|\bps{\tau}u\yhat-\bps{\tau}u\xhat_{\zhat}\right|\right\}  & \geq & 1\left\{ \lev a\rhat\geq\left|\bps{\tau}u\yhat-\bps{\tau}u\zhat\right|\right\}  & \geq & 1\left\{ \lev a\rhat-\rhat\geq\left|\bps{\tau}u\yhat-\bps{\tau}u\xhat_{\zhat}\right|\right\} ,\end{align*}
 and then Proposition \ref{pro:(Distortion-converges)} \eqref{enu: pro:(Distortion-converges) 2 BDP}\begin{align*}
1\left\{ \left(\lev a+1\right)\rhat\mK_{\bpsym}\geq\left|\yhat-\xhat_{\zhat}\right|\right\}  & \geq & 1\left\{ \lev a\rhat\geq\left|\bps{\tau}u\yhat-\bps{\tau}u\zhat\right|\right\}  & \geq & 1\left\{ \left(\lev a-1\right)\rhat/\mK_{\bpsym}\geq\left|\yhat-\xhat_{\zhat}\right|\right\} .\end{align*}
We integrate w.r.t. $d\mmc\left(\yhat\right)d\lev$ and then use the
$\zD$-set property \eqref{eq:D-set-property Ahlfors} of $\mmc$
for balls centered on $\xhat_{\zhat}$, \begin{align*}
\int_{0}^{1}C_{\mJ}\left(\left(\lev a+1\right)\rhat\mK_{\bpsym}\right)^{\zD}d\lev & \geq & \int_{\mJ}\A{\bps{\tau}u\yhat}{\bps{\tau}u\zhat}{\rhat}\, d\mmc\left(\yhat\right) & \geq & \int_{1/a}^{1}C_{\mJ}^{-1}\left(\left(\lev a-1\right)\rhat/\mK_{\bpsym}\right)^{\zD}d\lev.\end{align*}
 This inequality fits between $C_{\Along}\rhat^{\zD}$ and $C_{\Along}^{-1}\rhat^{\zD}$
for a suitable $C_{\Along}$.

To prove the fifth assertion, we will use $\left[\mf{\xplain|m}^{\prime}u\right]^{-1}\circ\mf{\xplain|m}=\bp{\xplain|m}u$
up to an additive constant (Proposition \ref{pro:(Distortion-converges)}
\eqref{enu:pro:(Distortion-converges) 3 similarity}): \begin{equation}
\zr{\xplain|m}u^{-1}\left|y-\mf{\xplain|m}\zhat\right|=\left|\left[\mf{\xplain|m}^{\prime}u\right]^{-1}\left(\mf{\xplain|m}\ms^{m}\yplain-\mf{\xplain|m}\zhat\right)\right|=\left|\bp{\xplain|m}u\circ\ms^{m}\yplain-\bp{\xplain|m}u\zhat\right|.\label{eq:lem:(neighborhood-net) Covariance-w.r.t-the distortion}\end{equation}
Inserting this into \eqref{eq:DEF A balls} results in the assertion.
\end{proof}

\subsection{Localizing and scaling the curvature under the inverse IFS\label{sub:Localizing-and-scaling}}

As $\mJ_{\ze}=\bigcup_{\tau\in\mI^{m}}\left(\mf{\tau}\mJ\right)_{\ze}$,
parts of $\mJ_{\ze}$ agree with $\left(\mf{\tau}\mJ\right)_{\ze}$.
Under the similarity $\mf{\tau}^{\prime}u$, this is the image of
another parallel set of $\bp{\tau}u\mJ$. For given $m\in\mathbb{N}$
and in a small area around $x\in\mJ$, the agreeing $\ze$ are less
than $\R\left(x,m\right)$. This idea also works for the associated
curvature measure instead of the set. Note $\hr$, $\hw$ below contain
the chunk of curvature between $\ze\leq\R\left(x,m\right)$ and $\ze\nleq\R\left(x,m+1\right)$. 

(The curvature of $\left(\bpsym\mJ\right)_{\ze}$ is a combination
of those of $\bpsym^{-1}\left(\left(\bpsym\mJ\right)_{\ze}\right)$,
compare \cite{Bernig03CurvatureTransformationMR1990086}. But this
does not help us.)

Note\begin{equation}
\hr_{\chunkh{\R\left(x,m+1\right)}{\R\left(x,m\right)}}^{k,\pm}\left(x,\ze,f\right)=\int\hi_{m}\left(\id,x,z,\ze\right)f\left(\ex z\right)\, dC_{k}^{\pm}\left(\mJ_{\ze},\ex z\right),\label{eq:hr in terms of hi}\end{equation}

\begin{equation}
\hw_{\tau}^{k,\pm}\left(x,\ze,f\right)=\int\hi_{0}\left(\bp{\tau}u,x,z,\ze\right)f\left(\ex z\right)\, dC_{k}^{\pm}\left(\left(\bp{\tau}u\mJ\right)_{\ze},\bp{\tau}u\ex z\right),\label{eq: hw in terms of hi}\end{equation}
 and with a placeholder $\bpsym$, \begin{equation}
\hi_{m}\left(\bpsym,x,z,\ze\right)\assign\chunk{\R\left(x,m+1\right)}{\R\left(x,m\right)}\left(\frac{\ze}{\left|\bpsym^{\prime}x\right|}\right)\,\frac{\ze^{\zD-k}\:\A{\bpsym x}{\bpsym z}{\ze}}{\int\left|\bpsym^{\prime}y\right|^{\zD}\A{\bpsym y}{\bpsym z}{\ze}\, d\mmc\left(y\right)}.\label{eq: DEF hi}\end{equation}

\begin{lem}
\label{lem:preimages of hr, hw curvdens}Given $\xplain\in\mJ$, $u\in\bV$,
$m\in\mathbb{N}$, the substitutions \begin{eqnarray*}
\xhat\assign\ms^{m}\xplain, & \tau\assign\xplain|m, & \rhat\assign\rplain/\zr{\tau}u\end{eqnarray*}
 transform $\hw$ and $\hr$ into each other: \begin{equation}
\hw_{\tau}^{k,\pm}\left(\xhat,\rhat,f\circ\mf{\tau}\right)=\hr_{\chunkh{\R\left(\xplain,m+1\right)}{\R\left(\xplain,m\right)}}^{k,\pm}\left(\xplain,\rplain,f\right).\label{eq:hw transforms-to hi chunk-of-curvdens}\end{equation}
 \end{lem}
\begin{proof}
\global\long\def\rfac{s}
Abbreviate $\hr\assign\hr_{\chunkh{\R\left(\xplain,m+1\right)}{\R\left(\xplain,m\right)}}^{k,\pm}\left(\xplain,\rplain,f\right)$.
For clarity, note that \[
\xplain=\mf{\tau}\xhat,\, m=\left|\tau\right|,\,\rplain=\rhat\zr{\xplain|m}u.\]
In the first half of the proof, we will rewrite $\hr$'s curvature
density $\hi_{m}\left(\id,\xplain,\zplain,\rplain\right)$ in terms
of the new variables\begin{eqnarray*}
\yhat\assign\mf{\tau}^{-1}\yplain, & \rhat=\rplain/\zr{\tau}u, & \ex{\zhat}=\mf{\tau}^{-1}\ex{\zplain}.\end{eqnarray*}
They will later leave the measure $\rplain^{-1}d\rplain$ unchanged.
This means $\xplain$, $\yplain$, $\ex{\zplain}$, and $\rplain$
are preimaged $m$ times under the IFS.

$\R\left(\xplain,m\right)$ was defined to be covariant: \[
1\left\{ \R\left(\xplain,m+1\right)<\rplain\leq\R\left(\xplain,m\right)\right\} =1\left\{ \zr{\tau}{\ms^{m}\xplain}\R\left(\ms^{m}\xplain,1\right)<\rhat\zr{\tau}u\leq\zr{\tau}{\ms^{m}\xplain}\R\left(\ms^{m}\xplain,0\right)\right\} \]
 \begin{equation}
=1\left\{ \R\left(\xhat,1\right)<\rhat\frac{\zr{\tau}u}{\zr{\tau}{\xhat}}\leq\R\left(\xhat,0\right)\right\} =1\left\{ \R\left(\xhat,1\right)<\frac{\rhat}{\zrp{\tau}u{\xhat}}\leq\R\left(\xhat,0\right)\right\} \label{eq:transformed R}\end{equation}
 In the denominator of $\hi_{m}\left(\id,\xplain,\zplain,\rplain\right)$,
Lemma \ref{lem:(neighborhood-net)} \eqref{enu:lem:(neighborhood-net) ball chain}
tells us \[
\yplain|m=\xplain|m.\]
Therefore, we can transform both instances of $\Along$ with the same
$\tau$, see \eqref{eq:lem:(neighborhood-net) Covariance-w.r.t-the distortion}
in the same lemma: \begin{eqnarray}
\A{\xplain}{\zplain}{\rplain}=\A{\xplain}{\mf{\tau}\zhat}{\rplain} & = & \A{\bp{\tau}u\xhat}{\bp{\tau}u\zhat}{\rhat}\,1\left\{ \tau=\xplain|m\right\} ,\label{eq:transformed Ax}\\
\A{\yplain}{\zplain}{\rplain}=\A{\yplain}{\mf{\tau}\zhat}{\rplain} & = & \A{\bp{\tau}u\yhat}{\bp{\tau}u\zhat}{\rhat}\,1\left\{ \tau=\yplain|m\right\} .\nonumber \end{eqnarray}
(A geometric interpretation: the covariance of $\Along$ was designed
to match that of the curvature.) We integrate the last two lines.
Then the Perron-Frobenius operator of $\mmc$ introduces $d\mmc\circ\ms^{-m}/d\mmc=\zr{\tau}{\yhat}^{\zD}$.
That is absorbed into $\bp{\tau}u$, \begin{eqnarray}
\int\A{\yplain}{\zplain}{\rplain}\, d\mmc\left(\yplain\right) & = & \int\A{\bp{\tau}u\ms^{m}\yplain}{\bp{\tau}u\zhat}{\rhat}1_{\left\{ \tau=\yplain|m\right\} }\, d\mmc\left(\yplain\right)\nonumber \\
 & = & \int\zr{\tau}{\yhat}^{\zD}\A{\bp{\tau}u\yhat}{\bp{\tau}u\zhat}{\rhat}\, d\mmc\left(\yhat\right)\nonumber \\
 & = & \int\zrp{\tau}u{\yhat}^{\zD}\A{\bp{\tau}u\yhat}{\bp{\tau}u\zhat}{\rhat}\, d\mmc\left(\yhat\right)\,\zr{\tau}u^{\zD}.\label{eq:transformed int A y}\end{eqnarray}
The remaining factor of $\hi_{m}\left(\id,\xplain,\zplain,\rplain\right)$
transforms as \[
\rplain^{\zD-k}=\rhat^{\zD-k}\zr{\tau}u^{\zD}\zr{\tau}u^{-k}.\]
 We assemble both factors above and \eqref{eq:transformed R}, \eqref{eq:transformed Ax},
\eqref{eq:transformed int A y} to obtain \begin{equation}
\zr{\tau}u^{k}\,\hi_{m}\left(\id,\xplain,\zplain,\rplain\right)=\hi_{0}\left(\bp{\tau}u,\xhat,\zhat,\rhat\right).\label{eq: integrand of hr in new variables}\end{equation}

In the second half of the proof, we will make use of the locality
\eqref{eq:Ck locality} and scaling \eqref{eq:Ck scaling} properties
of curvature measures to perform the remaining substitution $\ex{\zplain}\mapsto\bp{\tau}u\ex{\zhat}$.
First, we rename $\ex{\zplain}$ to $\mf{\tau}\ex{\zhat}$, \[
\hr=\int\hi_{m}\left(\id,\xplain,\zplain,\rplain\right)f\left(\mf{\tau}\ex{\zhat}\right)\, dC_{k}^{\pm}\left(\mJ_{\rplain},\mf{\tau}\ex{\zhat}\right),\]
To localize, only a small area around $\xplain$ supports $\zplain\mapsto\hi_{m}\left(\id,\xplain,\zplain,\rplain\right)$.
This is because $\Along$ lives on $\Int B\left(\xplain,a\R\left(\xplain,m\right)\right)$
(Lemma \ref{lem:(neighborhood-net)} \eqref{enu:lem:(neighborhood-net) Locality}),
which is disjoint from $\left(\mf{v|m}\mJ\right)_{\rplain}$ whenever
$\tau\assign\xplain|m\neq v|m\in\mI^{m}$ (Lemma \ref{lem:(Renewal-radii)}
\eqref{enu:lem:(Renewal-radii) 4 intersection mJ}). Therefore, we
may replace $\mJ_{\rplain}$ with $\left(\mf{\tau}\mJ\right)_{\rplain}$
inside the curvature measure:\[
\hr=\int\hi_{m}\left(\id,\xplain,\zplain,\rplain\right)f\left(\mf{\tau}\ex{\zhat}\right)\, dC_{k}^{\pm}\left(\left(\mf{\tau}\mJ\right)_{\rplain},\mf{\tau}\ex{\zhat}\right).\]
 Since $\bp{\tau}u\circ\mf{\tau}^{-1}$ is a similarity with Lipschitz
constant $\zr{\tau}u^{-1}$, it maps parallel sets to parallel sets,
\[
\bp{\tau}u\circ\mf{\tau}^{-1}\left(\left(\mf{\tau}\mJ\right)_{\rplain}\right)=\left(\bp{\tau}u\mJ\right)_{\rplain/\zr{\tau}u}=\left(\bp{\tau}u\mJ\right)_{\rhat}.\]
We may introduce the similarity via the scaling property \eqref{eq:Ck scaling}
of curvature measures, \begin{eqnarray*}
\hr & = & \zr{\tau}u^{k}\int\hi_{m}\left(\id,\xplain,\zplain,\rplain\right)f\left(\mf{\tau}\ex{\zhat}\right)\, dC_{k}^{\pm}\left(\bp{\tau}u\circ\mf{\tau}^{-1}\left(\left(\mf{\tau}\mJ\right)_{\rplain}\right),\bp{\tau}u\circ\mf{\tau}^{-1}\left(\mf{\tau}\ex{\zhat}\right)\right)\\
 & = & \zr{\tau}u^{k}\int\hi_{m}\left(\id,\xplain,\zplain,\rplain\right)f\circ\mf{\tau}\left(\ex{\zhat}\right)\, dC_{k}^{\pm}\left(\left(\bp{\tau}u\mJ\right)_{\rhat},\bp{\tau}u\ex{\zhat}\right).\end{eqnarray*}
(Scaling also proves $C_{k}^{\pm}\left(\left(\bp{\tau}u\mJ\right)_{\rhat},\cdot\right)$
is measurable as a function of $\rhat$ on the set where we need it,
see Remark \ref{rem:measurability}) This combines with \eqref{eq: integrand of hr in new variables}
to complete the proof.
\end{proof}

\subsection{Curvature of distorted images converges\label{sub:Curvature-of-distorted}}

In the last section, we compared $\mJ_{\ze}$ to a (locally) fixed
range of parallel sets of the distorted fractal $\bp{\widetilde{\omega|n}}u\mJ$.
Unlike the self-similar case, we cannot eliminate the distortion $\bp{\widetilde{\omega|n}}u$
altogether. But this map converges as $n\ra\infty$ and depends asymptotically
only on the {}``dynamical'' variable $\left(\omega,u\right)\in\mIN\times\mJ$.
In this section, we prove the associated {}``chunk of curvature''
$\hw_{x|n}^{\left(k,\pm\right)}$ converges, too.

The next continuity result will be applied to $\myset^{n}\assign\bp{\widetilde{\omega|n}}u\mJ$
and $\myset^{\infty}\assign\bplim{\omega}u\mJ$ for $u\in\mJ$, $\omega\in\mIN$.
If we had a corresponding result about the variations $C_{k}^{\pm}$,
our main theorem would hold for those, too (Remark \ref{rem:sharp}).
It seems difficult to quantify the rate of convergence unless the
reach can be controlled, compare \cite[Corollary 2, p. 384]{CohenMorvan06SecondFundamentalMeasureApproxCurvatureMR2269782}
in the smooth setting.
\begin{fact}
\label{fac:continuity of curvature}Let $\myset^{\infty}$, $\myset^{n}\subseteq\Rd$
be a sequence of nonempty, compact sets, and suppose $\myset^{n}$
converges to $\myset^{\infty}$ in the Hausdorff metric, i.e., \[
\lim_{n\ra\infty}\inf\left\{ \ze>0\,:\,\myset_{\ze}^{n}\supseteq\myset^{\infty}\,\text{and}\,\myset^{n}\subseteq\myset_{\ze}^{\infty}\right\} =0.\]

\begin{enumerate}
\item \label{enu:fac:continuity of curvature first-order}Let $k\in\left\{ d-1,d\right\} $.
For all $\ze>0$ except possibly a countable set, \[
\op{wlim}_{n\ra\infty}C_{k}\left(\widetilde{\myset_{\ze}^{n}},\,\cdot\times S^{d-1}\right)=C_{k}\left(\widetilde{\myset_{\ze}^{\infty}},\,\cdot\times S^{d-1}\right).\]

\item \label{enu:fac:continuity of curvature  second-order}Let $k\in\left\{ 0,\dots,d-1\right\} $,
and assume $\widetilde{\myset_{\ze}^{\infty}}$ has positive reach
for a fixed $\ze>0$. Then $\widetilde{\myset_{\ze}^{n}}$ has uniformly
positive reach for any sufficiently large $n$, and \[
\op{wlim}_{n\ra\infty}C_{k}\left(\widetilde{\myset_{\ze}^{n}},\,\cdot\right)=C_{k}\left(\widetilde{\myset_{\ze}^{\infty}},\,\cdot\right).\]

\end{enumerate}
\end{fact}
\begin{proof}
The first assertion is due to Stacho (\cite[Theorem 3]{Stacho76VolumeParallelSetsMR0442202}).
It was put in the present form by \cite[Theorem 5.1]{RatajSpodarevMeschenmoser09WienerSausageCurvatureMR2569809},
\cite[Corollary 2.5]{RatajWinter09MeasuresOfParallelSetsArxiv09053279}.
The second assertion can be found in \cite[Theorem 5.2]{RatajSpodarevMeschenmoser09WienerSausageCurvatureMR2569809}
for (non-directional) Federer curvatures only. The same work shows
$\liminf_{n\ra\infty}\op{reach}\widetilde{\myset_{\ze}^{n}}>0$. This
is enough for the normal current of $\widetilde{\myset_{\ze}^{n}}$
to converge in the flat semi-norms (\cite[Theorem 3.1]{RatajZaehle01MR1846894}).
Since Lipschitz-Killing curvature has a current representation (\cite[(18)]{Zaehle86IntRepMR849863}),
it converges weakly.
\end{proof}
Recall $\hi$ \eqref{eq: DEF hi} is a template for the integrands
of $\hw_{\tau}^{k,\pm}\left(x,\ze,f\right)$ \eqref{eq: hw in terms of hi}
for finite $\tau\in\mI^{*}$ and $\hw_{\widetilde{\omega}}^{k}\left(x,\ze,f\right)$
\eqref{eq:DEF hw limiting word} for infinite $\omega\in\mIN$, \[
\hi_{m}\left(\bpsym,x,z,\ze\right)=\chunk{\R\left(x,m+1\right)}{\R\left(x,m\right)}\left(\frac{\ze}{\left|\bpsym^{\prime}x\right|}\right)\,\frac{\ze^{\zD-k}\:\A{\bpsym x}{\bpsym z}{\ze}}{\int\left|\bpsym^{\prime}y\right|^{\zD}\A{\bpsym y}{\bpsym z}{\ze}\, d\mmc\left(y\right)}.\]
 
\begin{lem}
\label{lem:hw-converges} Let $\left(\omega,x\right)\in\mIN\times\mJ$,
$\ex z\in\bV\times S^{d-1}$, fix a reference point $u\in\bV$, and
assume $\ze>0$ is regular in the sense of Assumption \ref{ass:Regularity-of-parallel-sets}.
\begin{enumerate}
\item \label{enu:lem:hw-converges weak-limit hw}As a weak limit of signed
measures, \begin{equation}
\hw_{\widetilde{\omega}}^{k}\left(x,\ze,\cdot\right)=\op{wlim}\limits _{n\ra\infty}\hw_{\widetilde{\omega|n}}^{k}\left(x,\ze,\cdot\right).\label{eq:def hw plim}\end{equation}

\item \label{enu:lem:hw-converges integrand hi}The limit \[
\hi_{0}\left(\bplim{\omega}u,x,z,\ze\right)=\lim_{n\ra\infty}\hi_{0}\left(\bp{\widetilde{\omega|n}}u,x,z,\ze\right)\]
 converges uniformly in $z$ (possibly except at two values of $\ze$).
There is even a constant $c_{\hi}>0$ such that it converges at the
universal rate \[
c_{\hi}\left(1+\dist\left(\ze\zrplim{\omega}ux^{-1},\left\{ \R\left(x,1\right),\R\left(x,0\right)\right\} \right)^{-1}\right)\ze^{-k-1}\rmax^{n\holder}.\]

\end{enumerate}
\end{lem}
\begin{proof}
First, we will prove $\hi$ converges at the desired rate. It is enough
to know every factor and its denominator are bounded and converge
at such a rate (Fact \ref{fac:rate of convergence}). Both $\bp{\widetilde{\omega|n}}u^{\prime}x$
and $\bp{\widetilde{\omega|n}}ux$ converge at the rate $c_{\bpsym}\rmax^{n\holder}$
(Proposition \ref{pro:(Distortion-converges)} item \eqref{enu:pro:(Distortion-converges) 1 Both--converge}).
The factor \[
\chunk{\R\left(x,1\right)}{\R\left(x,0\right)}\left(\ze\zrp{\widetilde{\omega|n}}ux^{-1}\right)\:\,\ze^{-k}\]
is bounded by $\ze^{-k}$. Only at the endpoints of the interval can
it be discontinuous. To see how far its argument can stray, \[
\left|\ze\zrp{\widetilde{\omega|n}}ux^{-1}-\ze\zrplim{\omega}ux^{-1}\right|=\ze\zrplim{\omega}ux^{-1}\zrp{\widetilde{\omega|n}}ux^{-1}\,\left|\zrplim{\omega}ux-\zrp{\widetilde{\omega|n}}ux\right|\leq\ze\frac{c_{\bpsym}}{\mK^{2}}\rmax^{n\holder}.\]
 The bound $\ze^{-k}$ originating from the supremum norm at the worst
$n$ determines its overall rate $\frac{c_{\bpsym}\ze}{\mK^{2}\delta\ze^{k}}\rmax^{n\holder}$,
where $\delta$ is the distance from $\ze\zrp{\widetilde{\omega|n}}ux^{-1}$
to the nearest endpoint (which can be estimated by that of $\ze\zrplim{\omega}ux^{-1}$). 

The next factor $\A{\bp{\widetilde{\omega|n}}ux}{\bp{\widetilde{\omega|n}}uz}{\ze}$
is trivially bounded by $1$, and converges at a rate $\ze^{-1}c_{\bpsym}\rmax^{n\holder}$
due to $\op{Lip}\left(\Along\right)\leq\ze^{-1}$ (Lemma \ref{lem:(neighborhood-net)}
\eqref{enu:lem:(neighborhood-net) Lipschitz-continuity}). The estimate
(Lemma \ref{lem:(neighborhood-net)} \eqref{enu:lem:(neighborhood-net) D-set-property})
\[
c_{\bpsym}^{-1}C_{\Along}^{-1}\leq\frac{\int\left|\bpsym^{\prime}\right|^{D}\A{\bp{\widetilde{\omega|n}}ux}{\bp{\widetilde{\omega|n}}uz}{\ze}d\mmc\left(y\right)}{\ze^{\zD}}\leq c_{\bpsym}C_{\Along}\]
 makes sure we can divide without losing rate of convergence. Straightforward
reasoning treats the last formula's numerator. Convergence of $\hi_{0}\left(\bp{\widetilde{\omega|n}}u,x,z,\ze\right)$
is proved.

In preparation for the first item, let $\mu_{n}\ra\mu_{\infty}$ be
any sequence of weakly converging measures, and $g_{n}\ra g_{\infty}$
norm converging, continuous functions. Note $\sup_{m}\left\Vert \mu_{m}\right\Vert <\infty$
due to Banach-Steinhaus. Then \begin{equation}
\left|\mu_{n}\left(g_{n}\right)-\mu_{\infty}\left(g_{\infty}\right)\right|\leq\left|\left(\mu_{n}-\mu_{\infty}\right)\left(g_{\infty}\right)\right|+\sup_{m}\left\Vert \mu_{m}\right\Vert \left\Vert g_{n}-g_{\infty}\right\Vert _{\infty}\label{eq:duality separates convergence}\end{equation}
 permits us to treat the functions and measures separately. Back to
$\mu_{n}\assign\hw_{\widetilde{\omega|n}}^{k}\left(x,\ze,f\right)$,
we have proved its integrand $g_{n}\left(\ex z\right)\assign\hi_{0}\left(\bp{\widetilde{\omega|n}}u,x,z,\ze\right)f\left(\ex z\right)$
converges uniformly. Positive reach of $\left(\bplim{\omega}u\mJ\right)_{\ze}$
(Assumption \ref{ass:Regularity-of-parallel-sets} \eqref{enu:ass:Regularity-of-parallel-sets 3 reach distorted})
and Fact \ref{fac:continuity of curvature} (and equicontinuously
uniform convergence of $\bp{\widetilde{\omega|n}}u$ inside $d\ex z$)
confirm the curvature measure $\mu_{n}\left(d\ex z\right)\assign C_{k}\left(\left(\bp{\widetilde{\omega|n}}u\mJ\right)_{\ze},d\bp{\widetilde{\omega|n}}u\ex z\right)$
converges. (The point-wise limit $\hw_{\widetilde{\omega}}\left(x,\ze,f\right)$
remains measurable as a function of $\omega,x,\ze$, see Remark \ref{rem:measurability}).
\end{proof}

\subsection{Sufficient conditions for the uniform integrability assumption}

Our main result depends on a uniform integrability assumption. We
can always prove it for $C_{d}$ (Minkowski content) and $C_{d-1}$
(surface content) since these contain no principal curvatures. More
generally, Lemma \ref{lem:Zygmund space condition for uniform integrability}
below carefully estimates a dominating function. Its resulting order
of integrability is necessary in the self-similar case. The converse
dominated ergodic theorem \cite[Theorem 3.1.16]{Petersen89ErgodicTheoryMR1073173}
shows this when applied to \cite[(3.4.1)]{BohlZaehle12CurvatureDirectionMeasures1111.4457}
and the suspension flow.

Uniform integrability singles out $\mmbc$ and $\mmbi$ as the geometrically
{}``correct'' limit measures. Recall their thermodynamic potential
is the contraction ratio $x\mapsto-\zD\log\mf{x_{1}}\ms x$.\comment{V-variable fractals require a modified potential.}
Measures generated by arbitrary potentials are studied in multifractal
formalism \cite{Patzschke97ConformalMultifractalMR1479016}. At the
price of an extra multiplier, these could replace $\mmbc$ and $\mmbi$
up to here. But different potentials make mutually singular measures.
So fractal curvature can have a density with respect to at most one
of them.
\begin{lem}
\label{lem:Zygmund space condition for uniform integrability}The
uniform integrability assumption \eqref{eq:uniformly integrand Thm}
of the theorem holds if the function \begin{equation}
h\left(\ze,\omega,x\right)=\ze^{-k}\sup_{M>0}\frac{1}{M}\sum_{m=0}^{M-1}C_{k}^{\var}\left(\left(\bp{\widetilde{\omega\vert m}}u\mJ\right)_{\ze},\bp{\widetilde{\omega\vert m}}u\Int B\left(x,a\ze\right)\right)\,\chunk{\R\left(x,1\right)}{\R\left(x,0\right)}\left(\frac{\ze}{\zrp{\widetilde{\omega\vert m}}ux}\right)\label{eq:Zygmund space integrand}\end{equation}
 is in the Zygmund space, i.e., \[
\int_{\mIN\times\mJ}\int_{0}^{\infty}\max\left\{ 0,\left\vert h\left(\ze,\omega,x\right)\right\vert \ln\left\vert h\left(\ze,\omega,x\right)\right\vert \right\} \,\frac{d\ze}{\ze}d\mmbi\left(\omega,x\right)<\infty\]
 and if also \[
\int_{\mIN\times\mJ}\int_{0}^{\infty}\sup_{0<\Rvar\leq\frac{\Rstart}{2}}\frac{1_{\chunkh{\max\left\{ \R\left(x,0\right),\Rvar\right\} }{\Rstart}}\left(\ze\right)C_{k}^{\var}\left(\mJ_{\ze},\Int B\left(x,a\ze\right)\right)}{\ln\frac{\Rstart}{\Rvar}\,\ze^{k}}\,\frac{d\ze}{\ze}d\mmbi\left(\omega,x\right)<\infty.\]

\end{lem}
Note $\zrp{\widetilde{\omega\vert m}}xx=1$ simplifies $h$ in case
$u$ is replaced with $x$, see Theorem \ref{thm:Main result}.
\begin{proof}
We prepare to integrate a dominating function, i.e., the supremum
of \eqref{eq:uniformly integrand Thm} over all $\Rvar\in\left(0,\Rstart/2\right]$.
First, we eliminate $\Rvar$ in favor of the summation range $M$
of $m$. Besides $M$, only this subexpression of \eqref{eq:uniformly integrand Thm}
depends on $\Rvar$ explicitly: \[
M\assign\sum_{m\in\mathbb{N}}1\left\{ \ze\zr{\widetilde{\omega\vert m}}u>\Rvar\right\} \leq\sum_{m\in\mathbb{N}}1\left\{ \rmax^{m}>\frac{\Rvar}{\ze}\right\} \leq1+\ln\frac{\Rvar}{\ze}/\ln\rmax,\]
 \[
\frac{1}{M}\frac{M}{\ln\frac{\Rstart}{\Rvar}}\leq\frac{1}{M}\left(\frac{1}{\ln2}+\frac{1}{\left\vert \ln\rmax\right\vert }\frac{\ln\ze-\ln\Rvar}{\ln\Rstart-\ln\Rvar}\right)\leq\frac{1}{M}\left(c_{1}+c_{2}\left\vert \ln\ze\right\vert \right).\]
 Secondly, we can estimate $\hw_{\widetilde{\omega\vert m}}^{k,\var}$'s
integrand $\chunk{\R\left(x,1\right)}{\R\left(x,0\right)}\left(\frac{\ze}{\left|\bpsym^{\prime}x\right|}\right)\,\frac{\ze^{\zD-k}\:\A{\bpsym x}{\bpsym z}{\ze}}{\int\left|\bpsym^{\prime}y\right|^{\zD}\A{\bpsym y}{\bpsym z}{\ze}\, d\mmc\left(y\right)}$.
Lemma \ref{lem:(neighborhood-net)} item \eqref{enu:lem:(neighborhood-net) D-set-property}
{}``cancels'' $\ze^{\zD}$ with the denominator. Item \eqref{enu:lem:(neighborhood-net) Locality}
replaces $\Along$ in the numerator with the interior of its supporting
ball $B\left(x,a\ze\right)$: \[
\hw_{\widetilde{\omega\vert m}}^{k,\var}\left(x,\ze,1\right)\leq\chunk{\R\left(x,1\right)}{\R\left(x,0\right)}\left(\frac{\ze}{\zrp{\widetilde{\omega\vert m}}ux}\right)\, C_{\Along}\ze^{-k}C_{k}^{\var}\left(\left(\bp{\widetilde{\omega\vert m}}u\mJ\right)_{\ze},\bp{\widetilde{\omega\vert m}}u\Int B\left(x,a\ze\right)\right).\]
The estimate $\mK^{-1}\leq\zrp{\widetilde{\omega\vert m}}ux\leq\mK$
in the indicator function will help us find a finite measure $\tau$.
We apply the above preparations to \eqref{eq:uniformly integrand Thm}
and obtain the dominating function: \[
\sup_{0<\Rvar\leq\Rstart/2}\left[\frac{\hr_{\chunkh{\max\left\{ \R\left(x,0\right),\Rvar\right\} }{\Rstart}}^{k,\var}\left(x,\ze,1\right)}{\ln\frac{\Rstart}{\Rvar}}+\sum_{m\in\mathbb{N}}\frac{M}{\ln\frac{\Rstart}{\Rvar}}\frac{1\left\{ \ze\zr{\widetilde{\omega\vert m}}u>\Rvar\right\} }{M}\hw_{\widetilde{\omega\vert m}}^{k,\var}\left(x,\ze,1\right)\right]\]
 \[
\leq\left[\sup_{\Rvar}\frac{\hr_{\chunkh{\max\left\{ \R\left(x,0\right),\Rvar\right\} }{\Rstart}}^{k,\var}\left(x,\ze,1\right)}{\ln\frac{\Rstart}{\Rvar}}\right]+\left[\sup_{M}\frac{1}{M}\sum_{m<M}\left(c_{1}+c_{2}\left\vert \ln\ze\right\vert \right)\hw_{\widetilde{\omega\vert m}}^{k,\var}\left(x,\ze,1\right)\right]\]
\[
\leq\left[\sup_{\Rvar}\frac{1_{\chunkh{\max\left\{ \R\left(x,0\right),\Rvar\right\} }{\Rstart}}\left(\ze\right)C_{k}^{\var}\left(\mJ_{\ze},\Int B\left(x,a\ze\right)\right)}{\ln\frac{\Rstart}{\Rvar}\,\ze^{k}}\right]+\left[C_{\Along}\left(c_{1}+c_{2}\left\vert \ln\ze\right\vert \right)\, h\left(\ze,\omega,x\right)\right].\]
 Thirdly, the Young inequality (known from the duality of Orlicz spaces)
will be used to show the $h$-summand in the last line is integrable
w.r.t $\ze^{-1}d\ze d\mmbi\left(\omega,x\right)$. (Recall $\R\left(x,1\right)/\mK<\ze\leq\R\left(x,0\right)\mK$.)
The function $C_{\Along}\left(c_{1}+c_{2}\left\vert \ln\ze\right\vert \right)$
alone is {}``very'' integrable: By \eqref{eq:logdist is very integrable},
there is a $c_{3}>1$ such that \[
\int_{\mJ}\int_{\R\left(x,1\right)/\mK}^{\R\left(x,0\right)\mK}c_{3}^{\left\vert \ln\ze\right\vert }\frac{d\ze}{\ze}d\mmi\left(x\right)\leq\frac{c_{3}^{-\ln K}}{\ln c_{3}}\int_{\mJ}c_{3}^{-\ln\R\left(x,0\right)}d\mmi\left(x\right)<\infty.\]
 This also shows the measure $d\tau\left(\ze,\omega,x\right)\assign\ze^{-1}\chunk{\R\left(x,1\right)/\mK}{\R\left(x,0\right)\mK}\left(\ze\right)d\ze\, d\mmbi\left(\omega,x\right)$
is finite. We have concatenated $\ze\mapsto\left\vert \ln\ze\right\vert $
with $t\mapsto\Psi\left(t\right)\assign\min\left\{ t^{2},c_{3}^{t}\right\} $,
and $\Psi\left(\left\vert \ln\ze\right\vert \right)$ stays integrable.
The function $\Psi$ is convex. A straightforward computation shows
its Legendre transform $\Phi\left(t\right)\assign\int_{0}^{t}\inf\left\{ u:\Psi^{\prime}\left(u\right)>s\right\} ds$
satisfies $\lim_{t\ra\infty}\left(t\ln t\right)/\Phi\left(t\right)=\ln c_{3}$.
So $\Phi\left(\left\vert h\right\vert \right)$ is $\tau$-integrable
due to our assumption on $h$. The Young inequality for $\Psi$ and
$\Phi$ (see \textit{\cite{RaoRen91OrliczMR1113700}}) yields \[
\int\left\vert \ln\ze\right\vert h\left(\ze,\omega,x\right)d\tau\left(\ze,\omega,x\right)\leq\int\Phi\left(\left\vert h\right\vert \right)d\tau\,+\,\int\Psi\left(\left\vert \ln\ze\right\vert \right)d\tau<\infty.\]
 Due to linearity, $C_{\Along}\left(c_{1}+c_{2}\left\vert \ln\ze\right\vert \right)h\left(\ze,\omega,x\right)$
is $\tau$-integrable. So the dominating function is integrable.\end{proof}
\begin{lem}
\label{lem:sufficient conditions uniform integrability}The uniform
integrability assumption \eqref{eq:uniformly integrand Thm} of the
theorem holds if:
\begin{enumerate}
\item \label{enu:lem:sufficient-integrability essup}\[
c_{\sup}\assign\esup\limits _{\rplain>0,\xplain\in\mJ}\rplain^{-k}C_{k}^{\var}\left(\mJ_{\rplain},\Int B\left(\xplain,a\rplain\right)\right)<\infty\]
 ($\mmc$-essential in $\xplain$, Lebesgue in $\rplain$) or
\item \label{enu:lem:sufficient-integrability Minkowski} $k=d$ (Minkowski
content) or
\item \label{enu:lem:sufficient-integrability surface}$k=d-1$ (surface
content).
\end{enumerate}
\end{lem}
\begin{proof}
We will reduce the first assertion to the preceding Lemma \ref{lem:Zygmund space condition for uniform integrability}.
Define $\xplain\assign\mf{\widetilde{\omega|n}}\xhat$, $\rplain\assign\rhat\zr{\widetilde{\omega|n}}u$.
The proof of Lemma \ref{lem:preimages of hr, hw curvdens} (localizing
and scaling $\hw_{\widetilde{\omega\vert n}}^{k,\var}=\hr_{\dots}^{k,\pm}$)
shows \[
\rhat^{-k}C_{k}^{\var}\left(\left(\bp{\widetilde{\omega\vert m}}u\mJ\right)_{\rhat},\bp{\widetilde{\omega\vert m}}u\Int B\left(\xhat,a\rhat\right)\right)=\rplain^{-k}C_{k}^{\var}\left(\mJ_{\rplain},\Int B\left(\xplain,a\ze\right)\right)\leq c_{\sup}\]
 whenever $\rhat\leq\zrp{\widetilde{\omega\vert m}}u{\xhat}\R\left(\xhat,0\right)$.
This and then $\mK^{-1}\leq\zrp{\widetilde{\omega\vert m}}u{\xhat}\leq\mK$
imply for the preceding Lemma's $h$ \eqref{eq:Zygmund space integrand},
\[
h\left(\rhat,\omega,\xhat\right)\leq\sup_{M>0}\frac{1}{M}\sum_{m=0}^{M-1}c_{\sup}\,\chunk{\R\left(\xhat,1\right)}{\R\left(\xhat,0\right)}\left(\frac{\rhat}{\zrp{\widetilde{\omega\vert m}}u{\xhat}}\right)\leq c_{\sup}\,\chunk{\R\left(\xhat,1\right)/\mK}{\R\left(\xhat,0\right)\mK}\left(\rhat\right).\]
The measure $d\tau\left(\rhat,\omega,\xhat\right)\assign\rhat^{-1}\chunk{\R\left(\xhat,1\right)/\mK}{\R\left(\xhat,0\right)\mK}\left(\rhat\right)d\rhat\, d\mmbi\left(\omega,\xhat\right)$
is finite, see Lemma \ref{lem:(Renewal-radii)} \eqref{enu:lem:(Renewal-radii) 6 integrable}.
As a constant function, $c_{\sup}$ belongs to its Zygmund space.

The other assertions reduce to the first. In case $k=d$ (volume),
the supremum condition is trivial. The case $k=d-1$ (surface area)
is proved in \cite[Remark 3.2.2]{RatajZaehle10CurvatureDensitiesSelfSimilarSets}.
(Although stated only for self-similar $\mJ$, the proof works for
general compact sets.)
\end{proof}

\section{Appendix: Vector Toeplitz theorem}
\begin{lem}
\label{lemresummation}Let $\alpha_{n},\alpha_{\infty}$ be finite,
signed Borel measures on $\Rd$, $n\in\mathbb{N}$, such that the
weak{*}-limit (point-wise limit on continuous functions) \[
\op{wlim}\limits _{n\ra\infty}\alpha_{n}=\alpha_{\infty}\]
 exists. Let $b_{n},b_{\infty}$ be a bounded family of equicontinuous
functions whose averages converge in supremum norm \[
\lim_{N\ra\infty}\frac{1}{N}\sum_{n\leq N}b_{n}=b_{\infty}.\]
 Then the following limit exists: \[
\lim_{N\ra\infty}\frac{1}{N}\sum_{n\leq N}\alpha_{n}\left(b_{n}\right)=\alpha_{\infty}\left(b_{\infty}\right).\]
 \end{lem}
\begin{proof}
We will prove%
\footnote{A direct proof is easier than checking the topological assumptions
of \cite{CoxPowell69SummabilityLinearTopologicalSpacesMR0243235}.%
} the assertion in only one dimension for simplicity and because product
functions are dense. All measures shall be supported in $\left[x,y\right]$,
neglecting tails as the family is tight. Due to the Banach-Steinhaus
theorem, $\sup_{n}\left\Vert \alpha_{n}\right\Vert _{C^{\prime}\left[x,y\right]}<\infty$.
Since the quality of approximation is determined by the norm and the
modulus of continuity only (\cite[Thm. 1.2.10]{WernerFunkanaMR1787146}),
we can equiuniformly approximate the $b_{n}$ with Bernstein polynomials
$p_{t}$ using a fixed, finite set $T$ of points,\[
\left\Vert b_{n}-\sum_{t\in T}p_{t}b_{n}\left(t\right)\right\Vert _{C\left[x,y\right]}<\epsilon.\]
By combining both estimates, we will be finished once we can show
\[
\lim_{N\ra\infty}\frac{1}{N}\sum_{n\leq N}\alpha_{n}\left(\sum_{t\in T}p_{t}b_{n}\left(t\right)\right)=\alpha_{\infty}\left(\sum_{t\in T}p_{t}b_{\infty}\left(t\right)\right).\]
This reduces the assertion to the sequences of reals: consider $a_{n}\assign\alpha_{n}\left(p_{t}\right)$,
$b_{n}\left(t\right)$ for each point $t\in T$ separately. There,
it is a special case of Toeplitz's theorem \cite[Thm. 4.4.6]{WernerFunkanaMR1787146}
for the resummation method $1_{\left\{ n\leq N\right\} }b_{n}/N$
applied to $a_{n}$. 
\end{proof}

\section{Appendix: Example}

The group of Möbius transformations of $\Rd$ is generated by similarities
and the Möbius inversion $x\mapsto x/\left|x\right|^{2}$. The Möbius
transformation $\mf{\,}$ specified by its pole $p\in\Rd\cup\left\{ \infty\right\} $
and the derivative $rA$ ($r>0$, $A\in O\left(d\right)$) at its
fixed point $f\in\Rd$ is given by \begin{eqnarray}
\mf{\,}x & = & f+rA\left(x-f\right)+O\left(\left|x-f\right|^{2}\right)\nonumber \\
 & = & f+rA\left(\op{id}-2\op{Proj}_{\frac{p-f}{\left|p-f\right|}}\right)\left(\frac{\left|p-f\right|^{2}}{\left|p-x\right|^{2}}\left(x-p\right)+\left(p-f\right)\right),\label{eq:mobius transformation}\end{eqnarray}
 where $\op{Proj}_{y}$ denotes the orthogonal projection onto $y$.

The example in the introduction is generated by three Möbius maps.
They agree with those of the classical Sierpinski triangle up to first
order at their fixed points, except we reduce the contraction ratio
to re-obtain the Open Set Condition. All three maps $\mf i$ share
$r_{i}=785/2000$, $A_{i}=\id$ for $i\in\mI=\left\{ 1,2,3\right\} $.
But $f_{1}=\left(0,0\right)$, $p_{1}=\left(3,\frac{5}{2}\right)$
and $f_{2}=\left(1,0\right)$, $p_{2}=\left(-3,-2\right)$ and $f_{3}=\left(\frac{1}{2},\frac{\sqrt{3}}{2}\right)$,
$p_{3}=\left(-2,-4\right)$ differ.

\bibliographystyle{amsalpha}
\addcontentsline{toc}{section}{\refname}\bibliography{citations_bohl}

\end{document}